\documentclass[12pt]{amsart}
\usepackage{amsmath,amssymb}
\usepackage{calrsfs}
\usepackage{url}
\newcommand{\A}{{\mathbb{A}}}
\newcommand{\Z}{{\mathbb{Z}}}
\newcommand{\G}{{\mathbb{G}}}
\newcommand{\D}{{\mathbb{D}}}
\newcommand{\Q}{{\mathbb{Q}}}

\newcommand{\C}{{\mathbb{C}}}
\newcommand{\F}{{\mathbb{F}}}
\newcommand{\T}{{\mathbb{T}}}

\newcommand{\fX}{{\mathfrak{X}}}
\newcommand{\fU}{{\mathfrak{U}}}

\newcommand{\bq}{{\mathbf{q}}}
\newcommand{\bA}{{\mathbf{A}}}
\newcommand{\bB}{{\mathbf{B}}}
\newcommand{\cE}{{\mathcal{E}}}
\newcommand{\cC}{{\mathcal{C}}}
\newcommand{\cG}{{\mathcal{G}}}
\newcommand{\cH}{{\mathcal{H}}}
\newcommand{\cO}{{\mathcal{O}}}
\newcommand{\cP}{{\mathcal{P}}}
\newcommand{\cS}{{\mathcal{S}}}
\newcommand{\cT}{{\mathcal{T}}}
\newcommand{\cW}{{\mathcal{W}}}
\newcommand{\tB}{{\tilde{B}}}
\newcommand{\tG}{{\tilde{G}}}
\newcommand{\tT}{{\tilde{T}}}
\newcommand{\tF}{{\tilde{F}}}
\newcommand{\tZ}{{\tilde{Z}}}
\newcommand{\trho}{{\tilde{\rho}}}
\newcommand{\tta}{{\tilde{\eta}}}
\newcommand{\ttheta}{{\tilde{\theta}}}

\newcommand{\Aut}{{\operatorname{Aut}}}
\newcommand{\Ind}{{\operatorname{Ind}}}
\newcommand{\AV}{{\operatorname{AV}}}
\newcommand{\Irr}{{\operatorname{Irr}}}
\newcommand{\IBr}{{\operatorname{IBr}}}
\newcommand{\GL}{{\operatorname{GL}}}
\newcommand{\SL}{{\operatorname{SL}}}
\newcommand{\id}{{\operatorname{id}}}
\newcommand{\Gad}{G_{\operatorname{ad}}}
\newcommand{\Gder}{G_{\operatorname{der}}}
\newcommand\Rst{{^*\!R}}
\renewcommand{\leq}{\leqslant}
\renewcommand{\geq}{\geqslant}

\newtheorem{thm}{Theorem}[section]
\newtheorem{cor}[thm]{Corollary}
\newtheorem{lem}[thm]{Lemma}
\newtheorem{prop}[thm]{Proposition}

\theoremstyle{definition}
\newtheorem{defn}[thm]{Definition}
\newtheorem{exmp}[thm]{Example}
\newtheorem{rema}[thm]{}
\theoremstyle{remark}
\newtheorem{rem}[thm]{Remark}
\numberwithin{equation}{section}

\begin{document}

%\title[Irreducible characters and character sheaves]{On the irreducible 
%characters and character sheaves for reductive groups over finite fields}
\title[A first guide to the character theory of finite groups of Lie type]{A 
first guide to the character theory of finite groups of Lie type}
\author{Meinolf Geck}
\address{IAZ - Lehrstuhl f\"ur Algebra\\Universit\"at Stuttgart\\ 
Pfaffenwaldring 57\\D--70569 Stuttgart\\ Germany}
\email{meinolf.geck@mathematik.uni-stuttgart.de}

\date{May 14, 2017}
%\dedicatory{$\mbox{}$} 

\begin{abstract} This survey article is an introduction to some of Lusztig's 
work on the character theory of a finite group of Lie type $G(\F_q)$, where
$q$ is a power of a prime~$p$. It is partly based on two series of lectures 
given at the Centre Bernoulli (EPFL) in July 2016 and at a summer school in 
Les Diablerets in August 2015. Our focus here is on questions related to 
the parametrization of the irreducible characters and on results which 
hold without any assumption on~$p$ or~$q$.
\end{abstract}

\maketitle

\begin{center}
\begin{tabular}{lp{270pt}r} \multicolumn{3}{c}{\bf Contents} \\
1 & Introduction \dotfill & \pageref{sec0}\\
2 & The virtual characters of Deligne and Lusztig\dotfill & \pageref{secdl}\\
3 & Order and degree polynomials \dotfill & \pageref{secdeg}\\
4 & Parametrization of unipotent characters \dotfill & \pageref{secuni}\\
5 & Jordan decomposition (connected centre) \dotfill & \pageref{secj}\\
6 & Regular embeddings \dotfill & \pageref{secj1}\\
7 & Character sheaves \dotfill & \pageref{seccs}\\
8 & Appendix: On uniform functions \dotfill & \pageref{secapp}
\end{tabular}
\end{center}

%%%%%%%%%%%%%%%%%%%%%%%%%%%%%%%%%%%%%%%%%%%%%%%%%%%%%%%%%%%%%%%%%%%%%%%
\section{Introduction} \label{sec0}

According to Aschbacher \cite{As}, when faced with a problem about 
finite groups, it seems best to attempt to reduce the problem or a 
related problem to a question about simple groups or groups closely 
related to simple groups. The classification of finite simple groups 
then supplies an explicit list of groups which can be studied in 
detail using the effective description of the groups. In recent years,
this program has led to substantial advances on various long-standing open 
problems in the representation theory of finite groups; see, for example,
Malle's survey \cite{malle17}. The classification of finite simple groups 
itself highlights the importance of studying ``finite groups of Lie type'', 
which are the subject of our survey. 

So let $p$ be a prime and $k=\overline{\F}_p$ be an algebraic closure of 
the field with $p$ elements. Let $G$ be a connected reductive algebraic
group over $k$ and assume that $G$ is defined over the 
finite subfield $\F_q\subseteq k$, where $q$ is a power 
of~$p$. Let $F\colon G\rightarrow G$ be the corresponding Frobenius map.
Then the group of rational points $G^F=G(\F_q)$ is called a ``finite 
group of Lie type''. (For the basic theory of these groups, see \cite{C2}, 
\cite{mybook}, \cite{MaTe}, \cite{St67}.) We are interested in finding out 
as much information as possible about the complex irreducible 
characters of~$G^F$. 

In Section~\ref{secdl} we begin by recalling  some basic results about the 
virtual characters $R_{T,\theta}$ of Deligne-Lusztig \cite{DeLu}. In 
Section~\ref{secdeg} we explain the fact that the order of $G^F$ and 
the degrees of the irreducible characters of $G^F$ can be obtained by
evaluating certain polynomials at~$q$. The ``unipotent''
characters of $G^F$ form a distinguished subset of the set of 
irreducible characters of $G^F$. In Section~\ref{secuni} we describe, 
following Lusztig \cite{L10a}, a canonical bijection between the unipotent 
characters and a certain combinatorially defined set which only depends on 
the Weyl group $W$ of $G$ and the action of $F$ on~$W$. In 
Section~\ref{secj}, we expose some basic results from Lusztig's book
\cite{L1}, assuming that the centre $Z(G)$ is connected. The ``regular 
embeddings'' in Section~\ref{secj1} provide a technique to transfer results
from the connected centre case to the general case; see \cite{Lu5}, 
\cite{Lu08a}. Taken together, one obtains a full classification of the 
irreducible characters of $G^F$ (without any condition on $Z(G)$) 
including, for example, explicit formulae for the character degrees. 
Finally, in Section~\ref{seccs}, we discuss some basic features of
Lusztig's theory of character sheaves \cite{L2}. In \cite{L10}, Lusztig
closed a gap in this theory which now makes it possible to state a number
of results without any condition on $p$ or~$q$. 

As an application, and in response to a question from Pham Huu Tiep, we 
will then show that the results in \cite[\S 6]{gehi2} on the number of
unipotent $\ell$-modular Brauer characters of $G^F$ (for primes $\ell
\neq p$) hold unconditionally. In an appendix we show that, with all the 
methods available nowadays, it is relatively straightforward to settle 
an old conjecture of Lusztig \cite{Lu2} on ``uniform'' functions.

Our main references for this survey are, first of all, Lusztig's book 
\cite{L1}, and then the monographs by Carter \cite{C2}, Digne--Michel 
\cite{DiMi2}, Cabanes--Enguehard \cite{CE} and Bonnaf\'e \cite{Bo3} (in
chronological order). These five volumes contain a wealth of ideas, 
theoretical results and concrete data about characters of finite groups 
of Lie type. As far as character sheaves are concerned, we mostly rely on 
the original source \cite{L2}. We also recommend Lusztig's lecture 
\cite{L11} for an overview of the subject, as well as Shoji's older survey 
\cite{S5a}. The aim of the ongoing book project \cite{gema} is to provide a 
guided tour to this vast territory, where new areas and directions keep 
emerging; see, e.g., Lusztig's recent papers \cite{L12}, \cite{L13}. It is 
planned that a substantially expanded version of this survey will appear 
as Chapter~2 of~\cite{gema}.

\begin{rema} {\bf Notation.} \label{abs11} We denote by $\mbox{CF}(G^F)$ the 
vector space of complex valued functions on $G^F$ which are constant on the 
conjugacy classes of $G^F$. Given $f,f'\in \mbox{CF}(G^F)$, we denote by 
$\langle f,f' \rangle=|G^F|^{-1}\sum_{g \in G^F} f(g)\overline{f'(g)}$ the 
standard scalar product of $f,f'$ (where the bar denotes complex 
conjugation). Let $\Irr(G^F)$ be the set of complex irreducible characters 
of $G^F$; these form an orthonormal basis of $\mbox{CF}(G^F)$ with respect 
to the above scalar product. In the framework of \cite{DeLu}, \cite{L2}, 
class functions are constructed whose values are algebraic numbers in
$\overline{\Q}_\ell$ where $\ell\neq p$ is a prime. By choosing an embedding 
of the algebraic closure of $\Q$ in $\overline{\Q}_\ell$  into $\C$, we will 
tacitly regard these class functions as elements of $\mbox{CF}(G^F)$.
\end{rema}

Also note that we do assume that $G$ is defined over $\F_q$ and so we exclude
the Suzuki and Ree groups from the discussion. This certainly saves us from 
additional technical complications in the formulation of some results; it 
also makes it easier to give precise references. 
%(For example, it is not 
%obvious if all the results on character sheaves in Section~\ref{seccs} do 
%hold for endomorphisms $F\colon G \rightarrow G$ which are more general than 
%Frobenius maps.) 
Note that, in most applications to finite group theory (as 
mentioned above), the Suzuki and Ree groups can be regarded as ``sporadic 
groups'' and be dealt with separately. 

%%%%%%%%%%%%%%%%%%%%%%%%%%%%%%%%%%%%%%%%%%%%%%%%%%%%%%%%%%%%%%%%%%%%%%%
\section{The virtual characters of Deligne and Lusztig} \label{secdl}

Let $G,F,q$ be as in the introduction. The framework for dealing with 
questions about the irreducible characters of $G^F$ is provided by the 
theory originally developed by Deligne and Lusztig~\cite{DeLu}. In this
set-up, one associates a virtual character $R_{T,\theta}$ of $G^F$ to 
any pair $(T,\theta)$ where $T\subseteq G$ is an $F$-stable maximal 
torus and $\theta\in\Irr(T^F)$. An excellent reference for the definition
and basic properties is Carter's book \cite{C2}: orthogonality relations, 
dimension formulae and further character relations can all be found in 
\cite[Chap.~7]{C2}. We shall work here with a slightly different (but 
equivalent) model of $R_{T,\theta}$. (There is nothing new about this: it 
is already contained in \cite[1.9]{DeLu}; see also \cite[3.3]{Lu2}.) 

\begin{rema} \label{abs21}
Throughout, we fix an $F$-stable Borel subgroup $B_0\subseteq G$ and 
write $B_0=U_0 \rtimes T_0$ (semidirect product) where $U_0$ is the 
unipotent radical of $B_0$ and $T_0$ is an $F$-stable maximal torus. Let
$N_0:=N_G(T_0)$ and $W:=N_0/T_0$ be the corresponding Weyl group, with set
$S$ of simple reflections determined by $B_0$; let $l\colon W\rightarrow 
\Z_{\geq 0}$ be the corresponding length function. Then $(B_0,N_0)$ is a 
split $BN$-pair in $G$. Now $F$ induces an automorphism $\sigma\colon W
\rightarrow W$ such that $\sigma(S)=S$. By taking fixed points under~$F$, 
we also obtain a split $BN$-pair $(B_0^F, N_0^F)$ in the finite group $G^F$, 
with corresponding Weyl group $W^\sigma=\{w\in W\mid \sigma(w)=w\}$. (See 
\cite[\S 4.2]{mybook}.) If $w\in W$, then $\dot{w}$ always denotes a 
representative of $w$ in $N_G(T_0)$.

One advantage of the model of $R_{T,\theta}$ that we will now introduce
is that everything is defined with respect to the fixed pair $(B_0,T_0)$.
\end{rema}

\begin{rema} \label{abs22} For $w\in W$ we set 
$Y_{\dot{w}}:=\{x\in G\mid x^{-1}F(x)\in \dot{w}U_0\}\subseteq G$. Then
$Y_{\dot{w}}$ is a closed subvariety which is stable under left 
multiplication by elements of $G^F$. Now consider the subgroup 
\[ T_0[w]:=\{t\in T_0\mid F(t)=\dot{w}^{-1}t\dot{w}\}\subseteq T_0.\]
(Note that this does not depend on the choice of the representative 
$\dot{w}$.) One easily sees that $T_0[w]$ is a finite group; see, e.g.,
Remark~\ref{remrt} below. We check that $Y_{\dot{w}}$ is also stable under 
right multiplication by elements of $T_0[w]$. Indeed, let $t\in T_0[w]$ and 
$x\in Y_{\dot{w}}$. Then $F(t)=\dot{w}^{-1}t \dot{w}$ and so 
\[ (xt)^{-1}F(xt)=t^{-1}x^{-1}F(x)F(t)\in t^{-1}\dot{w}U_0F(t)=
\dot{w}\bigl(\dot{w}^{-1}t^{-1}\dot{w}U_0\dot{w}^{-1}t\dot{w}\bigr)=
\dot{w}U_0\]
since $T_0$ normalizes $U_0$. Consequently, by the same argument as
in \cite[\S 7.2]{C2}, a pair $(g,t)\in G^F\times T_0[w]$ induces linear 
maps on the $\ell$-adic cohomology spaces with compact support $H_c^i
(Y_{\dot{w}})$ ($i\in\Z$); furthermore, if $\theta \in \Irr(T_0[w])$, 
then we obtain a virtual character $R_w^\theta$ of $G^F$ by setting
\[ R_w^\theta(g):=\frac{1}{|T_0[w]|}\sum_{t \in T_0[w]}\sum_{i\in\Z}
(-1)^i \mbox{Trace}\bigl((g,t),H_c^i(Y_{\dot{w}})\bigr)\theta(t)\qquad 
(g\in G^F).\]
\end{rema}

\begin{rem} \label{remrt} Let $w\in W$. By Lang's Theorem (see, e.g.,
\cite[\S 1.17]{C2}), we can write $\dot{w}=g^{-1}F(g)$ for some $g\in G$.
Then $T:=gT_0g^{-1}$ is an $F$-stable maximal torus and $T_0[w]=g^{-1}T^F
g$. (In particular, $T_0[w]$ is finite.) Thus, $T$ is a torus ``of 
type $w$'' and may be denoted by~$T_w$. We define $\underline{\theta}\in
\Irr(T_w^F)$ by $\underline{\theta}(t):=\theta(g^{-1}tg)$ for $t\in 
T_w^F$. Then one checks that $R_w^\theta$ equals $R_{T_w,
\underline{\theta}}$ (as defined in \cite[\S 7.2]{C2}, \cite[2.2]{Lu2}) 
and $R_w^\theta$ does not depend on the choice of the 
representative~$\dot{w}$. (See, e.g., \cite[4.5.6]{mybook}.)
\end{rem}

\begin{exmp} \label{exprt} Let $w=1$. Then we can take $\dot{w}=g=1$ and
so $T_0[1]=T_0^F$. Since $F(U_0)=U_0$, Lang's Theorem implies that 
$Y_{\dot{1}}=\{yu\mid y\in G^F, u\in U_0\}$; furthermore, we obtain 
a surjective morphism $Y_{\dot{1}}\rightarrow (G/B_0)^F$, $x\mapsto xB_0$,
which is compatible with the actions of $G^F$ by left multiplication on
$Y_{\dot{1}}$ and on $(G/B_0)^F$. Let $Z_1,\ldots,Z_n$ be the fibres of 
this morphism, where $Z_1$ is the fibre of $B_0$ and $n=|(G/B_0)^F|$. Then
$G^F$ permutes these fibres transitively and the stabiliser of $Z_1$ is 
$B_0^F$. By one of the basic properties of $\ell$-adic cohomology listed 
in \cite[\S 7.1]{C2}, this already shows that $R_1^\theta$ is obtained
by usual induction from a virtual character of~$B_0^F$. Now $B_0^F=U_0^F
\rtimes T_0^F$ and so there is a canonical homomorphism $\pi\colon B_0^F
\rightarrow T_0^F$ with kernel $U_0^F$. Hence, if $\theta\in\Irr(T_0^F)$, 
then $\theta \circ \pi\in \Irr(B_0^F)$ and one further shows that 
(see the proof of \cite[7.2.4]{C2} for details):
\[R_1^\theta=\Ind_{B_0^F}^{G^F}\bigl(\theta\circ \pi\bigr).\]
So, in this case ($w=1$), $R_1^\theta$ is just ordinary induction of 
characters. Thus, in general, one may call $R_w^\theta$ ``cohomological
induction''.
\end{exmp}

\begin{prop} \label{dimform} We have $R_w^\theta(1)=
(-1)^{l(w)} q^{-N}|G^F:T_0[w]|$, where $N$ denotes the number of 
reflections in $W$. 
\end{prop}

\begin{proof} Write $R_w^\theta=R_{T_w,\underline{\theta}}$ as in 
Remark~\ref{remrt}. By \cite[7.5.1 and 7.5.2]{C2}, we have $R_{T_w,
\underline{\theta}}(1)=(-1)^{l(w)}|G^F: T_w^F|_{p'}$. It remains to note 
that $q^N$ is the $p$-part of $|G^F|$. (See also the formula for $|G^F|$ in
\ref{abs31} below.)
\end{proof}

Given $w,w'\in W$, we denote $N_{W,\sigma}(w,w'):=\{x\in W\mid xw
\sigma(x)^{-1}=w'\}$. Then a simple calculation shows that 
$\dot{x}T_0[w]\dot{x}^{-1}=T_0[w']$ if $x \in N_{W,\sigma}(w,w')$. 

\begin{prop} \label{scform} Let $w,w'\in W$, $\theta\in \Irr(T_0[w])$ and
$\theta'\in \Irr(T_0[w'])$. Then 
\[\langle R_w^\theta,R_{w'}^{\theta'}\rangle=|\{x\in N_{W,\sigma}(w,w')\mid 
\theta(t)=\theta'(\dot{x}t\dot{x}^{-1})\mbox{ for all $t \in T_0[w]$}\}|.\]
In particular, $R_w^\theta$ and $R_{w'}^{\theta'}$ are either equal or
orthogonal to each other.
\end{prop}

%$xw=w'\sigma(x)$ so $\sigma(x)w^{-1}=w'^{-1}x$.
%
%$F(\dot{x}t\dot{x}^{-1})=F(\dot{x})\dot{w}^{-1}t\dot{w}F(\dot{x})^{-1}=
%w'^{-1}xtx^{-1}w'$.  $F(\dot{x})=t'\sigma(x)$.
\begin{proof} Write $R_w^\theta=R_{T_w,\underline{\theta}}$ where
$T_w=gT_0g^{-1}$ and $g^{-1}F(g)=\dot{w}$, as in Remark~\ref{remrt}; 
similarly, we write $R_{w'}^{\theta'}=R_{T_{w'},\underline{\theta'}}$
where $T_{w'}=g'T_0g'^{-1}$ and $g'^{-1}F(g')=\dot{w}'$.
By \cite[Theorem 7.3.4]{C2}, the above scalar product is given by 
\begin{equation*}
|\{x\in G^F\mid xT_wx^{-1}=T_{w'} \mbox{ and } \underline{\theta}'
(xtx^{-1})=\underline{\theta}(t)\mbox{ for all $t\in T_w^F$}\}|/|T_w^F|.
\tag{$*$}
\end{equation*}
Hence, if $T_w$, $T_{w'}$ are not $G^F$-conjugate, then $N_{W,\sigma}
(w,w')=\varnothing$ (see \cite[3.3.3]{C2}) and both sides of the desired 
identity are~$0$. Now assume that $T_w$, $T_{w'}$ are $G^F$-conjugate. 
%Then $R_{T_w,\underline{\theta}}=R_{T_{w'},\underline{\theta}'}$ since
%\[\langle R_{T_w,\underline{\theta}},R_{T_{w},\underline{\theta}} \rangle=
%\langle R_{T_w,\underline{\theta}},R_{T_{w'},\underline{\theta}'} \rangle=
%\langle R_{T_{w'},\underline{\theta}'},R_{T_{w'},\underline{\theta}'} 
%\rangle.\]
Then we may as well assume that $w=w'$, $T_w=T_{w'}$, $g=g'$. Then ($*$) 
equals
\[|\{nT_w^F\in N_G(T_w)^F/T_w^F\mid \underline{\theta}'(ntn^{-1})=
\underline{\theta}(t) \mbox{ for all $t\in T_w^F$}\}.\]
By \cite[3.3.6]{C2}, $N_G(T_w)^F/T_w^F$ is isomorphic to the group 
$N_{W,\sigma}(w,w)$, where the isomorphism is given by sending 
$nT_w^F$ to the coset of $x_n:=g^{-1}ng \in N_G(T_0)$ mod~$T_0$. Now, for 
any $t \in T_w^F$, we have $\underline{\theta}(t)=\theta(g^{-1}tg)$ and 
$\underline{\theta}'(ntn^{-1})=\theta'(x_ng^{-1}tgx_n^{-1})$. This yields 
the desired formula. Finally note that, if ($*$) is non-zero, then
the pairs $(T_w,\underline{\theta})$ and $(T_{w'},\underline{\theta}')$
are $G^F$-conjugate, which implies that $R_{T_w,\underline{\theta}}=
R_{T_{w'},\underline{\theta}'}$.
\end{proof}

\begin{defn} \label{unif} We say that $f\in \mbox{CF}(G^F)$ is uniform
if $f$ can be written as a linear combination of $R_w^\theta$ for
various $w,\theta$. Let 
\[\mbox{CF}_{\text{unif}}(G^F):=\{f\in \mbox{CF}(G^F)\mid \mbox{$f$ is 
uniform}\}.\]
For example, by \cite[7.5]{DeLu}, \cite[12.14]{DiMi2}, the character of 
the regular representation of $G^F$ is uniform; more precisely, that
character can be written as 
\[ \frac{1}{|W|}\sum_{w\in W} \sum_{\theta\in\Irr(T_0[w])} 
R_w^\theta(1)R_w^\theta.\]
\end{defn}

\begin{prop} \label{dim2} Let $\rho\in \Irr(G^F)$. Then 
\[\rho(1)=\frac{1}{|W|}\sum_{w\in W} \sum_{\theta\in \Irr(T_0[w])}
\langle R_w^\theta,\rho\rangle R_w^\theta(1).\]
\end{prop}

\begin{proof} Just take the scalar product of $\rho$ with the above
expression for the character of the regular representation of $G^F$.
\end{proof}

\begin{rem} \label{green} Let $G_{\text{uni}}$ be the set of unipotent
elements of $G$. Let $w \in W$ and $\theta \in \Irr(T_0[w])$. If $u \in 
G_{\text{uni}}^F$, then $R_w^\theta(u) \in \Z$ and this value is independent 
of $\theta$; see \cite[7.2.9]{C2} and the remarks in \cite[\S 7.6]{C2}. 
Hence, we obtain a well-defined function $Q_w\colon G_{\text{uni}}^F
\rightarrow \Z$ such that $Q_w(u)=R_w^\theta(u)$ for $u \in G_{\text{uni}}^F$.
This is called the Green function corresponding to~$w$; see 
\cite[4.1]{DeLu}. There is a character formula (see \cite[7.2.8]{C2}) 
which expresses the values of $R_w^\theta$ in terms of $Q_w$, the 
values of $\theta$ and Green functions for various smaller groups. 
That formula immediately implies that 
\[ \sum_{\theta \in \Irr(T_0[w])} R_w^\theta(g)=\left\{\begin{array}{cl}
|T_0[w]|Q_w(g) & \quad \mbox{if $g\in G^F$ is unipotent}, \\ 0 & \quad
\mbox{otherwise}.\end{array}\right.\]
(This is also the special case $s_0=1$ in Lemma~\ref{app3} further 
below.) Using also Proposition~\ref{dimform}, we can re-write the identity in 
Proposition~\ref{dim2} as follows:
\[ \rho(1)=\frac{1}{|W|}\sum_{w\in W} |T_0[w]|\langle Q_w,\rho \rangle
Q_w(1)=\frac{1}{|W|}q^{-N}|G^F|\sum_{w\in W} (-1)^{l(w)}\langle 
Q_w,\rho \rangle,\]
where we regard $Q_w$ as a function on all of $G^F$, with value $0$
for $g \in G^F \setminus G_{\text{uni}}^F$. 
\end{rem}

\begin{rema} \label{dlgraph} By Proposition~\ref{dim2}, every irreducible
character of $G^F$ occurs with non-zero multiplicity in some $R_w^\theta$; 
furthermore, by Proposition~\ref{scform}, $R_w^\theta$ can have more than 
one irreducible constituent in general. This suggests to define a graph 
$\cG(G^F)$ as follows. It has vertices in bijection with $\Irr(G^F)$. Two 
characters $\rho \neq \rho'$ in $\Irr(G^F)$ are joined by an edge if 
there exists some pair $(w,\theta)$ such that $\langle R_w^\theta,\rho
\rangle \neq 0$ and $\langle R_w^\theta,\rho'\rangle\neq 0$. Thus, the 
connected components of this graph define a partition of $\Irr(G^F)$. There
are related partitions of $\Irr(G^F)$ into so-called ``geometric conjugacy 
classes'' and into so-called ``rational series'', but the definitions are
more complicated; see \cite[\S 12.1]{C2}, \cite[\S 10]{DeLu}, 
\cite[Chap.~13]{DiMi2}, \cite[\S 7]{L0a}. We just note at this point 
that the definitions immediately show the following implications: 
\begin{align*}
&\rho,\rho' \mbox{ belong to the same connected component of 
$\cG(G^F)$}\\
& \quad \Rightarrow \quad \mbox{$\rho,\rho'$ belong to the same 
rational series}\\
&\quad  \Rightarrow \quad \mbox{$\rho,\rho'$ belong to the same 
geometric conjugacy class}.
\end{align*}
We will clarify the relations between these notions in 
Theorem~\ref{agree1} and Remark~\ref{ratser} below. (See also 
Bonnaf\'e \cite[\S 11]{Bo3} for a further, detailed discussion).
\end{rema}
%Once we have developed more of
%the general theory, we will also indicate the relation of this partition
%with the above-mentioned ``rational series'' and ``geometric conjugacy
%classes''; see Theorem~\ref{agree1} and Remark~\ref{ratser} below (and
%Bonnaf\'e \cite[\S 11]{Bo3} for a further, detailed discussion).

\begin{exmp} \label{expsl2} Let $G=\SL_2(\overline{\F}_p)$, where $p\neq 2$,
and $F$ be the Frobenius map which raises each matrix entry to its 
$q$th power; then $G^F=\SL_2(\F_q)$. The character table of $G^F$ 
is, of course, well-known. (See, e.g., Fulton--Harris \cite[\S 5.2]{FuHa}.) 
The set $\Irr(G^F)$ consists of the following irreducible characters:
\begin{align*}
1_G &\quad \mbox{(trivial character)}; \\
\mbox{St}_G &\quad \mbox{with $\mbox{St}_G(1)=q$ (Steinberg character)};\\
\rho_i & \quad \mbox{with $\rho_i(1)=q+1$,  for $1\leq i \leq (q-3)/2$};\\
\pi_j & \quad \mbox{with $\pi_j(1)=q-1$,  for $1\leq j \leq (q-1)/2$};\\
\rho_0',\rho_0'' & \quad \mbox{each of degree $(q+1)/2$};\\
\pi_0',\pi_0'' & \quad \mbox{each of degree $(q-1)/2$}.
\end{align*}
Thus, in total, we have $|\Irr(G^F)|=q+4$. Let us re-interprete this in 
terms of the characters $R_w^{\theta}$; cf.\ Bonnaf\'e \cite[\S 5.3]{Bo4}.
For $\xi\in k^\times$, denote by $S(\xi)$ the diagonal matrix with diagonal 
entries $\xi,\xi^{-1}$. Then an $F$-stable maximal torus $T_0 \subseteq G$
as above and the corresponding Weyl group are given by
\[ T_0=\left\{S(\xi)\mid \xi \in k^\times\right\} \qquad \mbox{and}\qquad 
W=\langle s \rangle \quad \mbox{where} \quad \dot{s}=\begin{pmatrix} 0 & 1 \\
-1 & 0 \end{pmatrix}.\]
Let $\xi_0\in k^\times$ be a fixed element of order $q-1$ and 
$\xi_0'\in k^\times$ be a fixed element of order $q+1$. Then 
\begin{alignat*}{2}
T_0[1] &=T_0^F=\{S(\xi_0^a) \mid 0\leq a<q-1\} &&\quad \mbox{is cyclic of
order $q-1$};\\ T_0[s]&=\{S(\xi_0'^b)\mid 0\leq b<q+1\} &&\quad 
\mbox{is cyclic of order $q+1$}. 
\end{alignat*}
Let $\varepsilon\in \C$ be a primitive root of unity of order $q-1$;
let $\eta\in \C$ be a primitive root of unity of order $q+1$. For $i\in\Z$ 
let $\theta_i\in \Irr(T_0[1])$ be the character which sends $S(\xi_0^a)$ 
to $\varepsilon^{ai}$; for $j \in\Z$ let $\theta_j'\in \Irr(T_0[s])$ be the 
character which sends $S(\xi_0'^b)$ to $\eta^{bj}$. Then one finds that:
\[ R_{1}^{\theta_i}=\left\{\begin{array}{cl} 
1_{G}+\mbox{St}_{G} &  \mbox{if $i=0$},\\  
\rho_i &  \mbox{if $1\leq i\leq \frac{q{-}3}{2}$},\\
\rho_0'+\rho_0'' & \mbox{if $i=\frac{q{-}1}{2}$},\end{array}\right. 
\qquad R_{s}^{\theta_j'}=\left\{\begin{array}{cl} 
1_{G}-\mbox{St}_{G} & \mbox{if $j=0$},\\  
-\pi_j &  \mbox{if $1\leq j\leq \frac{q{-}1}{2}$},\\
-\pi_0'-\pi_0'' &  \mbox{if $j=\frac{q{+}1}{2}$};\end{array}\right.\]
furthermore, $R_1^{\theta_i}=R_1^{\theta_{-i}}$ and $R_s^{\theta_j'}=
R_s^{\theta_{-j}'}$ for $i,j\in \Z$. Hence, the graph $\cG(G^F)$ has 
$q+1$ connected components, which partition $\Irr(G^F)$ into the 
following subsets:
\begin{center}
$\{1_G,\mbox{St}_G\}, \quad \{\rho_i\} \;(1\leq i \leq \frac{q-3}{2}),
\quad \{\pi_j\} \;(1\leq j \leq \frac{q-1}{2}),\quad \{\rho_0', \rho_0''\},
\quad \{\pi_0',\pi_0''\}$.
\end{center}
The pairs $(1,\theta_i)$ and $(s,\theta_j')$, where $i=\frac{q-1}{2}$ and 
$j=\frac{q+1}{2}$, play a special role in this context; see 
Example~\ref{expsl2a} below.
\end{exmp}
% below for a further ) if $i=\frac{q-1}{2}$ and 
%. Hence, the set $\{\rho_0', \rho_0'',\pi_0',\pi_0''\}$ 
% are geometrically conjugate (in the sense of \cite[7.3.4]{C2};
%see also Example~\ref{expsl2a} below) if $i=\frac{q-1}{2}$ and 
%$j=\frac{q+1}{2}$. Hence, the set $\{\rho_0', \rho_0'',\pi_0',\pi_0''\}$ 

It is a good exercise to re-interprete similarly Srinivasan's character 
table \cite{Sr68} of $G^F=\mbox{Sp}_4(\F_q)$. (See also \cite{Sr94}.)

\begin{rem} \label{expsl2b} In the setting of 
Example~\ref{expsl2}, we see that $\dim \mbox{CF}_{\text{unif}}(G^F)=q+2$
and so $\mbox{CF}_{\text{unif}}(G^F)\subsetneqq \mbox{CF}(G^F)$.
Furthermore, we see that the two class functions
\begin{center}
$\frac{1}{2}(\rho_0'-\rho_0''+\pi_0'-\pi_0'') \qquad\mbox{and}\qquad
\frac{1}{2}(\rho_0'-\rho_0''-\pi_0'+\pi_0'')$
\end{center}
form an orthonormal basis of the orthogonal complement of 
$\mbox{CF}_{\text{unif}}(G^F)$ in $\mbox{CF}(G^F)$.
Using the character table of $G^F$ \cite[\S 5.2]{FuHa}, the values of
these class functions are given as follows, where  
$\delta=(-1)^{(q-1)/2}$. 
\[\renewcommand{\arraystretch}{1.2} \begin{array}{cccccc} \hline 
&J&J'&-J&-J'& \mbox{otherwise}\\\hline
%\frac{1}{2}(\rho_0'+\rho_0''+\pi_0'+\pi_0'') & q & \delta & 0 & 0 & \delta 
%& \delta & (-1)^a & -(-1)^b\\
%\frac{1}{2}(\rho_0'+\rho_0''-\pi_0'-\pi_0'') & 1 & \delta q & 1 & 1 & 0 & 
%0 & (-1)^a & (-1)^b\\
\frac{1}{2}(\rho_0'-\rho_0''+\pi_0'-\pi_0'') & 0 & 0 & \delta
\sqrt{\delta q} & -\delta\sqrt{\delta q} & 0 \\
\frac{1}{2}(\rho_0'-\rho_0''-\pi_0'+\pi_0'') & \sqrt{\delta q} & 
-\sqrt{\delta q} & 0 & 0 & 0 \\\hline\end{array}\]
Here, $J$ is $G^F$-conjugate to $\left(\begin{array}{cc} 1 & 1 \\ 0 & 1 
\end{array}\right)$ and $J'$ is $G^F$-conjugate to $\left(\begin{array}{cc} 
1 & \xi_0 \\ 0 & 1 \end{array}\right)$. We will encounter these two 
functions again at the end of this paper, in Example~\ref{expcusp}.
\end{rem}

\begin{exmp} \label{expgln} Assume that $G$ has a connected centre and 
that $W$ is either $\{1\}$ or a direct product of Weyl groups of type 
$A_n$ for various $n\geq 1$. Then it follows from Lusztig--Srinivasan
\cite{LuSr} that every class function on $G^F$ is uniform. In all other 
cases, we have $\mbox{CF}_{\text{unif}}(G^F)\subsetneqq \mbox{CF}(G^F)$,
and this is one reason why the character theory of $G^F$ is so much more
complicated in general.
\end{exmp}

In general, it seems difficult to say precisely how big 
$\mbox{CF}_{\text{unif}}(G^F)$ is inside $\mbox{CF}(G^F)$. To close this 
section, we recall a conjecture of Lusztig concerning this question, 
and indicate how a proof can be obtained by the methods that are available
now. (One can probably formulate a more complete answer in the framework of 
the results in Section~\ref{seccs}, but we will not pursue this here.)
Let $\cC$ be an $F$-stable conjugacy class of $G$; then $\cC^F$ 
is a union of conjugacy classes of~$G^F$. We denote by $f_{\cC}^G\in
\mbox{CF}(G^F)$ the characteristic function of $\cC^F$; thus, $f_{\cC}^G$ 
takes the value $1$ on $\cC^F$ and the value $0$ on the complement of~$\cC^F$. 

\begin{thm}[Lusztig \protect{\cite[Conjecture~2.16]{Lu2}}] \label{luconj} 
The function $f_{\cC}^G$ is uniform. 
\end{thm}

If $\cC$ is a unipotent class and $p$ is a ``good'' prime for~$G$, then 
this easily follows from the known results on Green functions; see Shoji 
\cite{S1}. It is shown in \cite[Prop.~1.3]{aver} that the condition on $p$ 
can be removed, using results from \cite{L2} and \cite[Theorem~5.5]{S3}. 
So Theorem~\ref{luconj} holds when $\cC$ is a unipotent class. By an 
argument analogous to that in \cite[25.5]{L2} (at the very end of Lusztig's 
character sheaves papers) one can deduce from this that Theorem~\ref{luconj} 
holds in complete generality; details are given in an appendix at the end 
of this paper (Section~\ref{secapp}).

%%%%%%%%%%%%%%%%%%%%%%%%%%%%%%%%%%%%%%%%%%%%%%%%%%%%%%%%%%%%%%%%%%%%%%%
\section{Order and degree polynomials} \label{secdeg}

A series of finite groups of Lie type (of a fixed ``type'') is an 
infinite family of groups like $\{\SL_n(\F_q)\mid \mbox{any $q$}\}$
(where $n$ is fixed) or $\{E_8(\F_q)\mid \mbox{any $q$}\}$. It then 
becomes meaningful to say that the orders of the groups in the family 
are given by a polynomial in~$q$, or that the degrees of the characters 
of these groups are given by polynomials in~$q$. In this section we 
explain how these polynomials can be formally defined. This will allow 
us to attach some useful numerical invariants to the irreducible characters 
of~$G^F$.

\begin{rema} \label{abs31} Let $X=X(T_0)$ be the character group of $T_0$, 
that is, the abelian group of all algebraic homomorphisms $\lambda\colon 
T_0 \rightarrow k^\times$; this group is free abelian of rank equal to 
$\dim T_0$. There is an embedding $W\hookrightarrow \Aut(X)$, $w\mapsto 
\underline{w}$, such that 
\[ \underline{w}(\lambda)(t)=\lambda(\dot{w}^{-1}t\dot{w}) \qquad
\mbox{for all $\lambda\in X$, $w\in W$, $t\in T_0$}.\]
Furthermore, $F$ induces a homomorphism $X \rightarrow X$, $\lambda \mapsto
\lambda \circ F$, which we denote by the same symbol. There exists 
$\varphi_0\in \Aut(X)$ of finite order such that
\[\lambda(F(t))=q\varphi_0(\lambda)(t)\quad \mbox{for all $t\in T_0$}.\]
(See \cite[1.4.16]{gema}.) Then the automorphism $\sigma \colon W
\rightarrow W$ in \ref{abs21} is determined by $\underline{w}\circ 
\varphi_0=\varphi_0\circ \underline{\sigma(w)}$ for all $w \in W$ (see,
e.g., \cite[6.1.1]{gema}). Now, the triple 
\[\G:=\bigl(X,\,\varphi_0,\,W \hookrightarrow \Aut(X)\bigr)\]
may be regarded as the combinatorial skeleton of $G,F$; note that the 
prime power~$q$ does not occur here. Let $\bq$ be an indeterminate 
over~$\Q$. For any $w\in W$, let us set 
\[ |\T_w|:=\det(\bq\,\id_X-\varphi_0^{-1}\circ \underline{w})\in \Z[\bq].\]
Note that $|\T_w|$ is monic of degree $\dim T_0$. By \cite[3.3.5]{C2}, the
order of the finite subgroup $T_0[w] \subseteq G$ introduced in 
Section~\ref{secdl} is given by evaluating $|\T_w|$ at $q$. In particular,
the order of $T_0^F$ is given by $|\T_1|(q)$. Let us now define 
\[ |\G|:=\bq^N |\T_1|\sum_{w \in W^\sigma}\bq^{l(w)}\in \Z[\bq],\]
where $N\geq 0$ denotes the number of reflections in $W$. Then the order 
of $G^F$ is given by evaluating $|\G|$ at~$q$; see, e.g., \cite[\S 2.9]{C2}, 
\cite[4.2.5]{mybook}. We call $|\G|\in \Z[\bq]$ the ``order polynomial''
of $G^F$. Note that $|\G|$ has degree $\dim G=2N+\dim T_0$ and that $\bq^N$ 
is the largest power of $\bq$ which divides $|\G|$. Since the order of 
$T_0[w]$ divides the order of $G^F$, one deduces that $|\T_w|$ divides 
$|\G|$ in $\Q[\bq]$. An alternative expression for $|\G|$ is given as 
follows. Steinberg \cite[Theorem~14.14]{St68} (see also \cite[3.4.1]{C2}) 
proves a formula for the total number of $F$-stable maximal tori in~$G$. 
This formula yields the identity
\[ |\G|=\bq^{2N}\Bigl(\frac{1}{|W|}\sum_{w\in W} \frac{1}{|\T_w|}
\Bigr)^{-1}.\]
(For further details, see \cite[\S 1]{BMM}, \cite[\S 1.6]{gema}, 
\cite[24.6, 25.5]{MaTe}.)
\end{rema}

\begin{rem} \label{abs31a} Let $Z=Z(G)$ be the centre of $G$ and 
$|\mathbb{Z}^\circ|\in\Q[\bq]$ be the order polynomial of the 
torus~$Z^\circ$. Then the discussion in \cite[\S 2.9]{C2} also shows that 
\[ |\T_1|=|\Z^\circ|\prod_{J \in S_\sigma} (\bq^{|J|}-1),\]
where $S_\sigma$ denotes the set of orbits of~$\sigma$ on~$S$. We call 
$|S_\sigma|$ the semisimple $\F_q$-rank of~$G$. Thus, $(\bq-1)^{|S_\sigma|}$
is the exact power of $\bq-1$ which divides the order polynomial of $\Gder$,
where $\Gder$ is the derived subgroup of $G$. Note that $G=Z^\circ.
\Gder$ and $Z^\circ\cap \Gder$ is finite; then $|G^F|=|(Z^\circ)^F|
|\Gder^F|$ and $|\G|=|\Z^\circ||\G_{\text{der}}|$ (see \cite[\S 2.9]{C2}).
\end{rem}

\begin{rema} \label{abs32} Now let us turn to the irreducible characters of 
$G^F$. In order to define ``degree polynomials'', we consider the virtual 
characters $R_w^\theta$ from the previous section. By 
Proposition~\ref{dimform} and the above formalism, we see that 
$R_w^\theta(1)$ is given by evaluating the polynomial $(-1)^{l(w)}
\bq^{-N}|\G|/|\T_w|\in \Q[\bq]$ at~$q$. Hence, setting 
\[\D_\rho:=\frac{1}{|W|}\sum_{w\in W}\sum_{\theta \in \Irr(T_0[w])}
(-1)^{l(w)} \langle R_w^\theta,\rho\rangle\,\bq^{-N}\frac{|\G|}{|\T_w|} 
\in \Q[\bq],\]
we deduce from Proposition~\ref{dim2} that $\rho(1)$ is obtained by 
evaluating the polynomial $\D_\rho$ at $q$; in particular, $\D_\rho\neq 0$.
Having $\D_\rho\in\Q[\bq]$ at our disposal, we obtain numerical invariants 
of $\rho$ as follows. 
\begin{align*} 
A_\rho\, &:=\mbox{ degree of $\D_\rho$},\\
a_\rho\, &:=\mbox{ largest non-negative integer such that $\bq^{a_\rho}$ 
divides $\D_\rho$},\\
n_\rho\, &:=\mbox{ smallest positive integer such that $n_\rho\D_\rho
\in\Z[\bq]$}.
\end{align*}
All we can say at this stage is that $0\leq a_\rho\leq A_\rho\leq N$ (since
$\bq^{-N}|\G|/|\T_w|\in\Q[\bq]$ has degree $N$); furthermore, $n_\rho$ 
divides $|W|$, since $|\T_w|\in\Z[\bq]$ is monic, $|\G|\in \Z[\bq]$ and, 
hence, $|W|\D_\rho\in\Z[\bq]$. In fact, it is known~---~but this requires 
much more work~---~that $\D_\rho$ always has the following form:
\begin{equation*}
\D_\rho=\frac{1}{n_\rho}\bigl(\bq^{A_\rho}+ \,\ldots \,\pm 
\bq^{a_\rho}\bigr), \tag{$\clubsuit$}
\end{equation*}
where the coefficients of all intermediate powers $\bq^i\,$ ($a_\rho<i<
A_\rho$) are integers and the number $n_\rho$ is typically much smaller
than the order of $W$. Indeed, if $Z(G)$ is connected, then this is 
contained in \cite[4.26]{L1} and \cite[\S 3]{Al}, \cite[\S 71B]{CR2}; for 
the general case, one uses an embedding of $G$ into a group with a
connected centre and the techniques described in Section~\ref{secj1} below 
(see Remark~\ref{abs32a}). We also mention that a formula like 
($\clubsuit$) already appeared early in Lusztig's work \cite[\S 8]{lue8}.
\end{rema}

\noindent {\bf Example.} Let $G^F=\SL_2(\F_q)$. Then $|\G|=\bq(\bq^2-1)$, 
$|\T_1|=\bq-1$ and $|\T_s|=\bq+1$, where we use the notation in 
Example~\ref{expsl2}. For $1\leq i\leq (q-3)/2$, the character $\rho_i$ 
occurs with multiplicity $1$ in $R_1^{\theta_i}$ and in $R_1^{\theta_{-i}}$. 
So $\D_{\rho_i}=\frac{1}{2}((\bq-1)+(\bq-1))=\bq-1$, as expected. Similarly, 
one finds that $\D_{1_G}=1$, $\D_{\text{St}_G}=\bq$, $\D_{\pi_j}=\bq+1$ 
(for $1\leq j \leq (q-1)/2$), $\D_{\rho_0'}=\D_{\rho_0''}=\frac{1}{2}(\bq-1)$
and $\D_{\pi_0'}= \D_{\pi_0''}= \frac{1}{2}(\bq+1)$.

%\begin{rem} \label{abs32aa} Let $f \in \Q[\bq]$ be a polynomial such that
%$f=\D_\rho$ for some $\rho \in \Irr(G^F)$. Then there is an infinite
%subset $E \subseteq \Z_{\geq 1}$ with the following properties. If $e \in E$,
%then 
%\begin{itemize}
%\item $\sigma$ is the automorphism of $W$ induced by $F^e$;
%\item $\varphi_0$ is the automorphism of finite order of $X$ induced by
%$F^e$;
%\item $f=\D_{\rho'}$ for some $\rho'\in \Irr(G^{F^e})$. 
%\end{itemize}
%\end{rem}

\begin{rema} \label{abs33} The integers $n_\rho$ and $a_\rho$ attached 
to $\rho$ can also be characterized more directly in terms of the character 
values of $\rho$, as follows. Let $\cO$ be an $F$-stable unipotent
conjugacy class of $G$. Then $\cO^F$ is a union of conjugacy classes of
$G^F$. Let $u_1,\ldots,u_r\in \cO^F$ be representatives of the classes 
of $G^F$ contained in $\cO^F$. For $1\leq i \leq r$ we set $A(u_i):=
C_G(u_i)/C_G^\circ(u_i)$. Since $F(u_i)=u_i$, the Frobenius map $F$ 
induces an automorphism of $A(u_i)$ which we denote by the same symbol. 
Let $A(u_i)^F$ be the group of fixed points under $F$. Then we set
\[ \AV(f,\cO):=\sum_{1 \leq i \leq r} |A(u_i):A(u_i)^F|f(u_i)\qquad
\mbox{for any $f\in \mbox{CF}(G^F)$}.\]
Note that this does not depend on the choice of the representatives $u_i$;
furthermore, the map $\mbox{CF}(G^F) \rightarrow \C$, $f \mapsto \AV(f,
\cO)$, is linear.

Now let $\rho\in\Irr(G^F)$ and set $d_\rho:=\max\{ \dim \cO \mid \AV(\rho,
\cO)\neq 0\}$. By the main results of \cite{GM1}, \cite{Lu6}, there is a 
unique $\cO$ such that $\dim \cO=d_\rho$ and $\AV(\rho,\cO)\neq 0$. This 
$\cO$ will be denoted by $\cO_\rho$ and called the unipotent support of 
$\rho$. Let $u\in \cO_\rho$. Then we have
\[ a_\rho=(\dim C_G(u)-\dim T_0)/2\qquad\mbox{and}\qquad 
\AV(\rho,\cO_\rho)=\pm \frac{1}{n_\rho}q^{a_\rho} |A(u)|.\]
(See \cite[Theorem~3.7]{GM1}.) Thus, from $\cO_\rho$ and $\AV(\rho,
\cO_\rho)$, we obtain $a_\rho$, $A(u)$ and, hence, also $n_\rho$. For 
further characterisations of these integers, see Lusztig \cite{L2007}.
\end{rema}

\begin{rem} The above results on the unipotent support of the irreducible 
characters of $G^F$ essentially rely on Kawanaka's theory \cite{Kaw1}, 
\cite{Kaw2} of ``generalized Gelfand--Graev representations'', which are 
defined only if $p$ is a ``good'' prime for $G$. Lusztig \cite{Lu6} showed 
how the characters of these representations can be determined assuming that
$p,q$ are sufficiently large. The latter restrictions on $p,q$ have been for 
a long time a drawback for applications of this theory. Recently, Taylor 
\cite{Tay15} has shown that Lusztig's results hold under the single 
assumption that~$p$ is good.
\end{rem}

\begin{exmp} \label{reguni} For any unipotent $u \in G$, we have 
$\dim C_G(u)\geq \dim T_0$. Furthermore, there is a unique unipotent 
class $\cO$ such that $\dim C_G(u)=\dim T_0$ for $u \in \cO$; this is 
called the regular unipotent class and denoted by $\cO_{\text{reg}}$. (See 
\cite[\S 5.1]{C2}.) In particular, $\cO_{\text{reg}}$ is $F$-stable.
Let $\rho \in \Irr(G^F)$. We say that $\rho$ is a ``semisimple character'' 
if $\AV(\rho,\cO_{\text{reg}})\neq 0$. These characters were originally 
singled out in the work of Deligne--Lusztig \cite[10.8]{DeLu} (see also 
\cite[\S 8.4]{C2}). Clearly, if $\rho$ is semisimple, then 
$\cO_{\text{reg}}$ is the unipotent support of $\rho$ in the sense of 
\ref{abs33}; furthermore, $\rho$ is semisimple if and only if $a_\rho=0$.

Semisimple characters appear ``frequently'' in the virtual characters
$R_w^\theta$. Indeed, let $w \in W$ and $\theta \in \Irr(T_0[w])$. 
By \cite[Theorem~9.16]{DeLu}, we have $R_w^\theta(u)=1$ for any $u \in 
\cO_{\text{reg}}^F$. This certainly implies that $\AV(R_w^\theta,
\cO_{\text{reg}}) \neq 0$ and so there exists some $\rho \in \Irr(G^F)$ 
such that $\langle R_w^\theta,\rho\rangle \neq 0$ and $\AV(\rho,
\cO_{\text{reg}})\neq 0$. Thus, every connected component of the graph 
$\cG(G^F)$ in \ref{dlgraph} has at least one vertex labeled by a semisimple
character.
\end{exmp}

\begin{rem} \label{agree} Assume that $Z(G)$ is connected. Then note that 
our definition of a ``semisimple character'' looks slightly different from 
that in \cite[p.~280]{C2}, \cite[10.8]{DeLu}, where the global average value 
on $\cO_{\text{reg}}^F$ is used instead of our $\AV(\rho,\cO_{\text{reg}})$. 
Let us check that the two definitions do agree. First note that
$\cO_{\text{reg}}\cap U_0^F\neq \varnothing$; furthermore, if $u\in 
\cO_{\text{reg}}\cap U_0$, then $C_G(u)=Z(G).C_{U_0}(u)\subseteq T_0.U_0=
B_0$ (see \cite[14.15]{DiMi2}). Hence, a standard application of 
Lang's Theorem shows that we can find elements $u_1,\ldots,u_r \in U_0^F$
which form a set of representatives of the classes of $G^F$ inside 
$\cO_{\text{reg}}^F$. By \cite[14.15, 14.18]{DiMi2}, the group $A(u_i)$ is 
generated by the image of $u_i$ in $A(u_i)$; hence, we have $A(u_i)^F=
A(u_i)$ and so 
\[ \AV(\rho,\cO_{\text{reg}})=\rho(u_1)+\ldots + \rho(u_r).\]
On the other hand, as above, we have $C_G(u_i)=Z(G).C_{U_0}(u_i)$ and so 
$|C_G(u_i)^F|=|Z(G)^F||C_{U_0}^\circ(u_i)^F||A(u_i)|$. Now 
$C_{U_0}^\circ(u_i)$ is a connected unipotent group and so 
$|C_{U_0}^\circ(u_i)^F|=q^{d_i}$ where $d_i:=\dim C_{U_0}^\circ(u_i)$; 
see, e.g., \cite[4.2.4]{mybook}. Also note that all $d_i$ are equal and 
that all the groups $A(u_i)$ have the same order. Hence, $|C_G(u_i)^F|$ 
does not depend on $i$. So the global average value is given by 
\[ \sum_{g \in \cO_{\text{reg}}^F} \rho(g)=\sum_{1 \leq i \leq r}
|G^F:C_G(u_i)^F|\,\rho(u_i)=c\bigl(\rho(u_1)+\ldots +\rho(u_r)\bigr)
\quad \mbox{where $c\neq 0$}.\]
Thus, indeed, the global average value is non-zero if and only if $\AV(\rho,
\cO_{\text{reg}})\neq 0$. (Note that, in general, the global average value
of $\rho$ on an $F$-stable unipotent class $\cO$ will not be proportional 
to $\AV(\rho,\cO)$.)
\end{rem}

\begin{rema} \label{parab} Let $\cP(S)$ be the set of all subsets of $S$.
Then $\sigma\colon W\rightarrow W$ (see \ref{abs21}) acts on $\cP(S)$ and we
denote by $\cP(S)^\sigma$ the $\sigma$-stable subsets. For $J \in 
\cP(S)^\sigma$ we have a corresponding $F$-stable parabolic subgroup $P_J=
\langle B_0,\dot{s} \mid s \in J \rangle \subseteq G$. This has a Levi 
decomposition $P_J=U_J \rtimes L_J$ where $U_J$ is the unipotent radical of 
$P_J$ and $L_J$ is an $F$-stable closed subgroup such that $T_0\subseteq 
L_J$; furthermore, $L_J$ is connected reductive, $B_0\cap L_J$ is an 
$F$-stable Borel subgroup of $L_J$, and the Weyl group $N_{L_J}(T_0)/T_0$ 
of $L_J$ is isomorphic to the parabolic subgroup $W_J:=\langle J \rangle$ 
of~$W$. We have $P_J^F=U_J^F\rtimes L_J^F$ and so there is canonical 
homomorphism $\pi_J\colon P_J^F\rightarrow L_J^F$ with kernel $U_J^F$. 
Hence, if $\phi \in \Irr(L_J^F)$, then $\phi \circ \pi_J\in \Irr(P_J^F)$ 
and we define 
\[R_J^S(\phi):=\Ind_{P_J^F}^{G^F}\bigl(\phi\circ \pi_J\bigr) \qquad
\mbox{``Harish-Chandra induction''}.\]
We say that $\rho \in \Irr(G^F)$ is cuspidal if $\langle R_J^S(\phi),
\rho\rangle=0$ for any $J\in \cP(S)^\sigma\setminus \{S\}$ and any $\phi 
\in \Irr(L_J^F)$. If $J\in \cP(S)^\sigma$ and $\phi\in \Irr(L_J^F)$ is 
cuspidal, then let 
\[\Irr(G^F|J,\phi):=\{ \rho \in \Irr(G^F)\mid \langle R_J^S(\phi),
\rho\rangle\neq 0\}.\]
With this notation, $\Irr(G^F)$ is the union of the sets $\Irr(G^F|J,
\phi)$ for various $(J,\phi)$. (See \cite[Chap.~9]{C2}, 
\cite[\S 70B]{CR2}, \cite[Chap.~6]{DiMi2}.) The following result 
strengthens earlier results of Howlett--Lehrer which are described 
in detail in \cite[Chap.~10]{C2}.
\end{rema}

\begin{prop}[Lusztig \protect{\cite[8.7]{L1}}] \label{luhc} Assume that 
$Z(G)$ is connected. Let $J\in \cP(S)^\sigma$ and $\phi\in \Irr(L_J^F)$ be 
cuspidal. Let $W_J(\phi)$ denote the stabilizer of $\phi$ in 
$N_G(L_J)^F/L_J^F$. Then $W_J(\phi)$ is a finite Weyl group and there is a 
natural bijection 
\[\Irr(W_J(\phi)) \stackrel{\sim}{\longrightarrow}\Irr(G^F|J,\phi), 
\qquad \epsilon \mapsto \phi[\epsilon],\]
(depending on the choice of a square root of $q$ in $\C$).
\end{prop}

The characters of finite Weyl groups are well-understood; see, e.g., 
\cite{GePf}. So the above result provides an effective parametrization
of the characters in $\Irr(G^F|J,\phi)$. The situation is technically
more complicated when the centre $Z(G)$ is not connected; see, e.g.,
\cite[Chap.~10]{C2}. (Note that one complication has disappeared in the
meantime: the $2$-cocyle appearing in \cite[10.8.5]{C2} is always
trivial; see \cite{myhc}.)

\begin{rem} \label{parab1} Let $\rho \in \Irr(G^F)$ and $J\in \cP(S)^\sigma$
be such that $\langle R_J^S(\phi),\rho \rangle\neq 0$ for some cuspidal
$\phi \in \Irr(L_J^F)$. As in Remark~\ref{abs31a}, let $S_\sigma$ be the
set of orbits of $\sigma$ on~$S$; similarly, let $J_\sigma$ be the set of 
orbits of $\sigma$ on~$J$. Then the number $t:=|S_\sigma|-|J_\sigma|$ 
may be called the ``depth'' of $\rho$, as in \cite[4.1]{L76a}. It is known
(see, e.g., \cite[9.2.3]{C2}) that $t$ is well-defined. Thus, the 
``depth'' provides a further numerical invariant attached to $\rho$. For
example, $\rho$ has depth $0$ if and only if $\rho$ is cuspidal.
\end{rem}

\begin{rem} \label{abs3fin} One can show that the degree polynomials 
$\D_\rho$ in \ref{abs32} behave in many ways like true character degrees. 
We leave it as a challenge to the interested reader to prove the following 
statements for $\rho\in \Irr(G^F)$:
\begin{itemize}
\item[(a)] $\D_\rho$ divides $|\G|$.
\item[(b)] If $\langle R_w^\theta,\rho\rangle\neq 0$, then $\D_\rho$
divides $|\G|/|\T_w|$. 
\item[(c)] If $d=\max\{i \geq 0 \mid (\bq-1)^i \mbox{ divides $\D_\rho$}\}$,
then $\rho$ has depth $|S_\sigma|-d\geq 0$; in particular, $\rho$ is
cuspidal if and only if $(\bq-1)^{|S_\sigma|}$ divides~$\D_\rho$.
\end{itemize}
(For hints, see \cite[\S 2]{myblock}, \cite[\S 2]{BMM}. If $q,p$ are very 
large, then the above properties are equivalent to analogous properties of 
actual character degrees. In order to reduce to this case, one can use 
an argument as in the proof of \cite[Theorem~3.7]{GM2}.)
\end{rem}

%%%%%%%%%%%%%%%%%%%%%%%%%%%%%%%%%%%%%%%%%%%%%%%%%%%%%%%%%%%%%%%%%%%%%%%
\section{Parametrization of unipotent characters} \label{secuni}

We say that $\rho \in \Irr(G^F)$ is unipotent if $\langle R_w^1,\rho
\rangle\neq 0$ for some $w \in W$, where $1$ stands for the trivial 
character of $T_0[w]$. We set
\[\fU(G^F)=\{\rho \in \Irr(G^F)\mid \rho \mbox{ unipotent}\}.\]
We will see that these play a distinguished role in the theory.

\begin{rem} \label{uni1} For any $w \in W$, we have $\langle R_w^1,1_G
\rangle=1$ and $\langle R_w^1,\mbox{St}_G\rangle=(-1)^{l(w)}$, where $1_G$ 
denotes the trivial character of $G^F$ and $\mbox{St}_G$ denotes the 
Steinberg character of $G^F$ (see \cite[7.6.5, 7.6.6]{C2}). Thus, we have 
$1_G\in \fU(G^F)$ and $\mbox{St}_G \in \fU(G^F)$. In fact, it is even
true that $1_G$ and $\mbox{St}_G$ are uniform:
\[ 1_G=\frac{1}{|W|} \sum_{w \in W} R_w^1 \qquad \mbox{and}\qquad
\mbox{St}_G=\frac{1}{|W|} \sum_{w \in W} (-1)^{l(w)}R_w^1;\]
see \cite[12.13, 12.14]{DiMi2}. We also have $\langle R_w^\theta,\rho
\rangle=0$ if $\rho \in \fU(G^F)$ and $\theta \neq 1$; see 
\cite[\S 12.1]{C2}. Hence, $\fU(G^F)$ defines a connected component of 
the graph $\cG(G^F)$ in \ref{dlgraph}.
\end{rem}

By Lusztig's Main Theorem~4.23 in \cite{L1}, the classification of 
$\fU(G^F)$ only depends on the pair $(W,\sigma)$. We shall now describe
this classification, where we do not follow the scheme in \cite{L1} but 
that in \cite[\S 3]{L10a}. This will be done in several steps.

\begin{rema} \label{abs41} Let us begin by explaining how the 
classification of $\fU(G^F)$ is reduced to the case where $G$ is a simple 
algebraic group of adjoint type. (See \cite[1.18]{L76a} and \cite[3.15]{Lu2}).
First, since $G/Z(G)$ is semisimple, there exists a surjective homomorphism 
of algebraic groups $\pi \colon G\rightarrow \Gad$ which factors through 
$G/Z(G)$ and where $\Gad$ is a semisimple group of adjoint type (see 
\cite[1.5.8]{gema} and \cite[p.~45/64]{St67}). 
Furthermore, there exists a Frobenius map $F\colon \Gad\rightarrow\Gad$ 
(relative to an $\F_q$-rational structure on $\Gad$) such that $F\circ\pi
=\pi\circ F$; thus, $\pi$ is defined over $\F_q$. (See
\cite[1.5.9(b)]{gema} and \cite[9.16]{St68}.) Hence, we obtain a group
homomorphism $\pi \colon G^F\rightarrow \Gad^F$ (but note that this is 
not necessarily surjective). By \cite[7.10]{DeLu}, this induces a bijection  
\[ \fU(\Gad^F)\stackrel{\sim}{\longrightarrow} \fU(G^F),\qquad
\rho \mapsto \rho\circ \pi.\] 
Now we can write $\Gad=G_1\times \ldots \times G_r$ where each $G_i$ 
is semisimple of adjoint type, $F$-stable and $F$-simple, that is, 
$G_i$ is a direct product of simple algebraic groups which are cyclically 
permuted by $F$. Let $h_i\geq 1$ be the number of simple factors in
$G_i$, and let $H_i$ be one of these. Then $F^{h_i}(H_i)=H_i$ and 
\[ \iota_i \colon H_i\rightarrow G_i, \qquad g \mapsto gF(g)\ldots 
F^{h_i-1}(g),\]
is an injective homomorphism of algebraic groups which restricts to
an isomorphism $\iota_i \colon H_i^{F_i}\stackrel{\sim}{\longrightarrow}
G_i^F$ where we denote $F_i:=F^{h_i}|_{H_i}\colon H_i\rightarrow H_i$. 
(See \cite[1.5.15]{gema}.) Let $f\colon G_1\times \ldots \times G_r
\rightarrow \Gad$ be the product map. Then, finally, it is shown in
\cite[1.18]{L76a} that $f$ and the homomorphisms $\iota_1,\ldots,
\iota_r$  induce bijections
\[\fU(\Gad^F)\stackrel{\sim}{\longrightarrow} \fU(G_1^F)\times \ldots 
\times \fU(G_r^F)\stackrel{\sim}{\longrightarrow} \fU(H_1^{F_1})\times
\ldots \times \fU(H_r^{F_r}).\] 
Thus, the classification of $\fU(G^F)$ is reduced to the case where 
$G$ is simple of adjoint type. 
\end{rema}

\begin{rema} \label{abs43} In order to parametrize the set $\fU(G^F)$,
we need some further invariants attached to the unipotent characters
of $G^F$. (For example, the invariants $A_\rho$, $a_\rho$, $n_\rho$ in 
\ref{abs33} are not sufficient.) For this purpose, we use an alternative 
characterization of the unipotent characters of $G^F$. This is based on
the varieties (see \cite[\S 1]{DeLu})
\[X_w:=\{ gB_0\in G/B_0\mid g^{-1}F(g)\in B_0\dot{w}B_0\} \qquad (w\in W).\]
Note that $X_w$ is stable under left multiplication by elements of $G^F$.
Hence, any $g\in G^F$ induces a linear map of $H_c^i(X_w)$ ($i\in\Z$). We 
have a morphism of varieties $Y_{\dot{w}}\rightarrow X_w$, $x\mapsto xB_0$,
which turns out to be surjective. By studying the fibres of this morphism 
and by using some basic properties of $\ell$-adic cohomology with compact 
support, one shows that 
\begin{center}
$R_w^1(g)=\sum_i (-1)^i\mbox{Trace}(g,H_c^i(X_w))\qquad\mbox{for all
$g\in G^F$};$
\end{center}
see \cite[7.7.8, 7.7.11]{C2}. Now let $\delta\geq 1$ be the order
of $\sigma\in\Aut(W)$. Then $F^\delta(X_w)=X_w$ and, hence, $F^\delta$ 
induces a linear map of $H_c^i(X_w)$ which commutes with the linear maps
induced by the elements of $G^F$. Consequently, if $\mu
\in\overline{\Q}_\ell$ is an eigenvalue of $F^\delta$ on $H_c^i(X_w)$, 
then the corresponding generalized eigenspace $H_c^i(X_w)_\mu$ is a 
$G^F$-module. Now every $\rho\in\fU(G^F)$ occurs as a constituent of 
$H_c^i(X_w)_\mu$ for some~$w$, some~$i$ and some~$\mu$. By Digne--Michel 
\cite[III.2.3]{DiMi0} and Lusztig \cite[3.9]{Lu2}, there is a well-defined 
root of unity $\omega_\rho$ with the following property. If $\rho$ occurs 
in $H_c^i(X_w)_\mu$ for some $w,i,\mu$, then $\mu=\omega_\rho q^{m\delta/2}$ 
where $m\in \Z$. (Here, we assume that a square root $q^{1/2}\in
\overline{\Q}_\ell$ has been fixed.) We call $\omega_\rho$ the Frobenius 
eigenvalue of $\rho$. We shall regard $\omega_\rho$ as an element of~$\C$
(via our chosen embedding of the algebraic numbers in $\overline{\Q}_\ell$ 
into $\C$; see \ref{abs11}).
\end{rema}

\begin{rema} \label{abs42}
Assume that $G/Z(G)$ is simple or $\{1\}$. Then $W=\{1\}$ or $W\neq \{1\}$
is an irreducible Weyl group. We now define a subset $\fX^\circ(W,\sigma)
\subseteq \C^\times\times\Z$, which only depends on the pair $(W,\sigma)$. 
Assume first that $\sigma$ is the identity; then we write $\fX^\circ(W)=
\fX^\circ(W,\id)$. If $W=\{1\}$, then $\fX^\circ(W)=\{(1,1)\}$. Now 
let $W\neq \{1\}$.  Then the sets $\fX^\circ(W)$ are given as follows.
\begin{itemize}
\item Type $A_n$ ($n \geq 1$): $\fX^\circ(W)= \varnothing$.
\item Type $B_n$ or $C_n$ ($n \geq 2$): $\fX^\circ(W)=\{((-1)^{n/2},
2^l)\}$ if $n= l^2+l$ for some integer $l \geq 1$, and $\fX^\circ(W)=
\varnothing$ otherwise.
\item Type $D_n$ ($n \geq 4$): $\fX^\circ(W)=\{((-1)^{n/4}, 
2^{2l-1}\}$ if $n=4l^2$ for some integer $l \geq 1$, and 
$\fX^\circ(W)= \varnothing$ otherwise.
\item Type $G_2$: $\fX^\circ(W)=\{(1,6),(-1,2),
(\theta,3), (\theta^2,3)\}$.
\item Type $F_4$: $\fX^\circ(W)=\{(1,8),(1,24),
(-1,4), (\pm i, 4), (\theta,3),(\theta^2,3)\}$.
\item Type $E_6$: $\fX^\circ(W)=\{(\theta,3), (\theta^2,3)\}$.
\item Type $E_7$: $\fX^\circ(W)=\{(\pm i,2)\}$.
\item Type $E_8$: $\fX^\circ(W)=\{(1,8)$, $(1,120)$,
$(-1,12)$, $(\pm i, 4)$, $(\pm \theta,6)$, $(\pm \theta^2,6)$,
$(\zeta,5)$, $(\zeta^2,5)$, $(\zeta^3,5)$, $(\zeta^4,5)\}$.
\end{itemize}
Here, $\theta$, $i$, $\zeta\in\C$ denote fixed primitive roots of unity
of order $3$, $4$, $5$, respectively. 

Now assume that $\sigma$ is not the identity; let $\delta\geq 2$ be the
order of $\sigma$. Then the sets $\fX^\circ(W,\sigma)$ are given as follows.
\begin{itemize}
\item Type $A_n$ ($n\geq 2$) and $\delta=2$:  
$\fX^\circ(W,\sigma)=\{((-1)^{\lfloor (n+1)/2\rfloor},1)\}$ if 
$n+1=l(l-1)/2$ for some integer $l \geq 1$, and $\fX^\circ(W,\sigma)=
\varnothing$ otherwise.
\item Type $D_n$ ($n\geq 4$) and $\delta=2$: $\fX^\circ(W,\sigma)=
\{(1,2^{2l})\}$ if $n=(2l+1)^2$ for some integer 
$l \geq 1$, and $\fX^\circ(W,\sigma)=\varnothing$ otherwise.
\item Type $D_4$ and $\delta=3$: $\fX^\circ(W,\sigma)=\{(\pm 1,2)\}$.
\item Type $E_6$ and $\delta=2$: $\fX^\circ(W,\sigma)=\{(1,6),
(\theta,3),(\theta^2,3)\}$.
\end{itemize} 
Let us denote $\fU^\circ(G^F):=\{\rho \in \fU(G^F) \mid \rho 
\mbox{ cuspidal}\}$. Now we can state:
\end{rema}

%\medskip
%{\bf Step 3}. Assume that $W$ is not irreducible but $\sigma$-irreducible,
%that is, $W$ is the direct product of $h\geq 2$ irreducible factors
%which are cyclicly permuted by $\sigma$. Then we set 
%$\fS_{W,\sigma}^\circ:=\fS_{W',\sigma^h}^\circ$ where $W'$ is an
%irreducible direct factor of $W$. (Note that $\sigma^h(W')=W'$ and that 
%all direct factors of $W$ are isomorphic Weyl groups.)
%
%\medskip
%{\bf Step 4}. In general, we can write $W=W_1\times \ldots \times W_r$ 
%where $W_i$ are $\sigma$-stable Weyl groups which are 
%$\sigma_i$-irreducible, where $\sigma_i$ denotes the restriction of 
%$\sigma$ to $W_i$. Then we set $\fS_{W,\sigma}^\circ:=
%\fS_{W_1,\sigma_1}^\circ\times \ldots \times \fS_{W_r,\sigma_r}^\circ$.
%

\begin{thm}[Lusztig \protect{\cite{L10a}}] \label{thm1} Assume that 
$G/Z(G)$ is simple or $\{1\}$. There exists a unique bijection
$\fX^\circ(W,\sigma)\stackrel{\sim}{\longrightarrow} \fU^\circ(G^F)$ with 
the following property. If $\rho \in \fU^\circ(G^F)$ corresponds to 
$x=(\omega,m)\in \fX^\circ(W,\sigma)$, then $\omega=\omega_\rho$ 
(see \ref{abs43}), $m=n_\rho$ (see \ref{abs32}).
\end{thm}

\begin{rem} (a) The second component of the pairs in $\fX^\circ(W,\sigma)$ 
is only needed to distinguish the two cuspidal unipotent characters in types
$F_4$, $E_8$ which both have Frobenius eigenvalue~$1$. Lusztig 
\cite[3.3]{L10a}, \cite[\S 18]{L11} uses slightly different methods
to achieve this distinction. A similar unicity statement is 
contained  in \cite[\S 6]{DiMi1}.

(b) The proof of Theorem~\ref{thm1} relies on the explicit knowledge of 
the degree polynomials $\D_\rho$ and of the Frobenius eigenvalues 
$\omega_\rho$ in all cases. Tables with the degrees of the cuspidal 
unipotent characters can be found in \cite[13.7]{C2}. If $\sigma=\id$, 
then $\omega_\rho$ is determined in \cite[11.2]{L1}. If $\sigma\neq \id$, 
then this can be extracted from \cite[7.3]{L76a}, \cite[3.33]{Lu2}, 
\cite[\S 4]{GM2}. (The case where $W$ is of type $A_n$ and $\delta=2$ can 
also be dealt with by a similar argument as in the proof of \cite[4.11]{GM2}.)
The explicit knowledge of the degree polynomials also shows that 
the function $\rho \mapsto a_\rho$ is constant on $\fU^\circ(G^F)$.
\end{rem}

\begin{rema} \label{abs44} Next we use the fact that the classification
of $\fU(G^F)$ can be reduced to the classification of $\fU^\circ(G^F)$,
using the concept of Harish-Chandra induction as in \ref{parab}. (See 
Lusztig \cite[3.25]{Lu2} for a detailed explanation of this reduction.)
Thus, if $G/Z(G)$ is simple, then we have a partition
\[\fU(G^F)=\bigsqcup_{(J,\phi)} \Irr(G^F|J,\phi);\]
here, the union runs over {\it all} pairs $(J,\phi)$ such that $J \in
\cP(S)^\sigma$ and $\phi \in \fU^\circ(L_J^F)$; also note that, if 
$\fU^\circ(L_J^F)\neq \varnothing$, then $L_J/Z(L_J)$ is simple or 
$\{1\}$. Furthermore, the stabiliser $W_J(\phi)$ in Proposition~\ref{luhc} 
now has a more explicit description:
\[ W_J(\phi)=\cW^{S/J}:=\{w\in W\mid\sigma(w)=w\mbox{ and } wJw^{-1}=J\};\]
this is a Weyl group with simple reflections in bijection with the orbits 
of $\sigma$ on $S\setminus J$ (see \cite[8.2, 8.5]{L1}). As before, we 
have a natural bijection
\[\Irr(\cW^{S/J}) \stackrel{\sim}{\longrightarrow} \Irr(G^F|J,
\phi), \qquad \epsilon \mapsto \phi[\epsilon],\]
(depending on the choice of a square root of $q$ in $\C$). Now let 
\[ \fX(W,\sigma):=\{(J,\epsilon,x)\mid J \in \cP(S)^\sigma, \; 
\epsilon \in \Irr(\cW^{S/J}), \; x\in \fX^\circ(W_J,\sigma)\}.\]
We have an embedding $\fX^\circ(W,\sigma)\hookrightarrow
\fX(W,\sigma)$, $x\mapsto (S,1,x)$. 
\end{rema}

\begin{cor}[Lusztig \protect{\cite{L10a}}] \label{cor1} Assume that 
$G/Z(G)$ is simple or $\{1\}$. There exists a unique bijection 
$\fX(W,\sigma)\stackrel{\sim}{\longrightarrow} \fU(G^F)$ with the following 
property. If $\rho \in \fU(G^F)$ corresponds to $(J,\epsilon,x)\in 
\fX(W,\sigma)$, then $\rho=\phi[\epsilon]$ where $\phi\in 
\fU^\circ(L_J^F)$ corresponds to~$x$ under the bijection in 
Theorem~\ref{thm1}.
\end{cor}

In this picture, those $\rho \in \fU(G^F)$ which occur in $R_1^1$ (the 
character of the permutation module $\C[G^F/B_0^F]$) correspond to triples 
$(J,\epsilon,x)$ where $J=\varnothing$, $x=(1,1)$ and $\epsilon\in 
\Irr(W^\sigma)$. (Note that $\cW^{S/\varnothing}=W^\sigma$.) If $\epsilon$
is the trivial character, then $\rho=1_G$ is the trivial character of~$G^F$; 
if $\epsilon$ is the sign character, then $\rho=\mbox{St}_G$ is the 
Steinberg character of~$G^F$.  (See, e.g., \cite[\S 68B]{CR2}.) At the other
extreme, the cuspidal unipotent characters of $G^F$ correspond to triples 
$(J,\epsilon,x)$ where $J=S$, $\epsilon=1$ and $x \in \fX^\circ(W,\sigma)$.
(Note that $\cW^{S/S}=\{1\}$.)

\begin{rema} \label{abs45} If $G/Z(G)$ is not simple or $\{1\}$, then 
consider the reduction arguments in \ref{abs41}. Thus, we obtain a 
natural bijection
\[\fU(G^F) \stackrel{\sim}{\longrightarrow} \fU(H_1^{F_1})\times \ldots 
\times \fU(H_r^{F_r}),\]
where each $H_i\subseteq \Gad$ is a simple algebraic group and $F_i=
F^{h_i}|_{H_i} \colon H_i\rightarrow H_i$ for some $h_i \geq 1$. Let $W_i$ 
be the Weyl group of $H_i$ and $\sigma_i$ be the automorphism of $W_i$ 
induced by $F_i$. Then Corollary~\ref{cor1} yields a bijection
\begin{equation*}
\fU(G^F) \stackrel{\sim}{\longrightarrow} \fX(W,\sigma):=
\fX(W_1,\sigma_1)\times \ldots \times \fX(W_r,\sigma_r).\tag{a}
\end{equation*}
By \cite[Main Theorem~4.23]{L1}, there exist integers $m_{\hat{x},w}\in \Z$ 
(for $\hat{x} \in \fX(W,\sigma)$ and $w \in W$), which only depend on the 
pair $(W,\sigma)$, with the following property. If $\rho \in \fU(G^F)$ 
corresponds to $\hat{x} \in \fX(W,\sigma)$ under the bijection in (a), then 
\begin{equation*}
\langle R_w^1,\rho\rangle=m_{\hat{x},w} \qquad \mbox{for all $w \in W$}.
\tag{b}
\end{equation*}
Furthermore, there are explicit formulae for $m_{\hat{x},w}$ in 
terms of Lusztig's ``non-abelian Fourier matrices'' (which first appeared
in \cite[\S 4]{lue8}), the function $\Delta\colon \fX(W,\sigma)\rightarrow
\{\pm 1\}$ in \cite[4.21]{L1}, and the (``$\sigma$-twisted'') character 
table of $W$. Finally, let $\hat{x}_0\in \fX(W,\sigma)$ correspond to the
trivial character $1_G \in\fU(G^F)$. Then  
\begin{equation*}
\mbox{$\hat{x}_0$ is uniquely determined by the condition that $m_{\hat{x}_0,
w}=1$ for all $w \in W$}.\tag{c}
\end{equation*} 
Indeed, $\langle R_w^1,1_G\rangle=1$ for all $w \in W$; see 
Remark~\ref{uni1}. On the other hand, if we also have $m_{\hat{x},w}=1$ 
for some $\hat{x} \in \fX(W,\sigma)$, then $\langle R_w^1,\rho\rangle=1$ 
for the corresponding $\rho\in \fU(G^F)$. But then $\langle 1_G,\rho
\rangle=\frac{1}{|W|}\sum_{w \in W} \langle  R_w^1,\rho\rangle=1$ and so
$\rho=1_G$, hence $\hat{x}=\hat{x}_0$. 
\end{rema}

\begin{rem} \label{abs46} For $\rho \in \fU(G^F)$, we denote by
$\Q(\rho)=\{\rho(g) \mid g \in G^F\} \subseteq \C$ the character field
of $\rho$. Then $\Q(\rho)$ is explicitly known in all cases; see
\cite{myschur}, \cite{Lurat}. Assume that $G/Z(G)$ is simple.
Let $J\in \cP(S)^F$ and $\phi\in \fU^\circ(L_J^F)$ be cuspidal such
that $\rho \in \fU(G^F|J,\phi)$. Then $\Q(\rho) \subseteq \Q(q^{1/2},
\omega_\phi)$, where $q^{1/2}$ is only needed for certain $\rho$ for 
types $E_7,E_8$; see \cite[5.4, 5.6]{myschur} for further details.
\end{rem}

\begin{rem} \label{mult1} One may wonder what general statements about 
the multiplicities $m_{\hat{x},w}$ could be made. For example, is it true
that, for any $\rho \in \fU(G^F)$, there exists some $w \in W$ such that
$\langle R_w^1,\rho\rangle=\pm 1$~? In \cite[p.~356]{L1}, there is an 
example of a cuspidal unipotent character (for the Ree group of type 
${^2\!F}_4$) which has even multiplicity in all $R_w^1$. There are also
examples (e.g., in type $C_4$) of non-cuspidal unipotent characters which 
have even multiplicity in all $R_w^1$; see Lusztig \cite[2.21]{Lurat}.
\end{rem}
%see \cite{} for all cases where $\Q(\rho)=\Q$ and further references there. 
%Let $G^F=E_7(q)$ and $\rho\in \fU^\circ(G^F)$ corresponding
%to pair $(\pm i,2)\in \fX_{W,\id}^\circ$.
%Ohmori???

%
%\begin{prop} \label{field} Let $\rho \in \fU(G^F)$. Then 
%$\omega_\rho \in \Q(\rho)$.????
%\end{prop}
%
%????\cite[3.9]{L10a}: similar bijection with unipotent character sheaves,
%then in \cite[3.10]{L10a}: canonical bijection between $\fU(G^F)$
%and unipotent character sheaves. Much more conceptual in \cite{L12}.
%Also \cite{L13}.

%%%%%%%%%%%%%%%%%%%%%%%%%%%%%%%%%%%%%%%%%%%%%%%%%%%%%%%%%%%%%%%%%%%%%%%
\section{Jordan decomposition (connected centre)} \label{secj}

The Jordan decomposition reduces the problem of classifying the irreducible 
characters of $G^F$ to the classification of unipotent characters for 
certain ``smaller'' groups associated with $G$. If the centre $Z(G)$ is 
connected, then this is achieved by Lusztig's Main Theorem~4.23 in 
\cite{L1}; in \cite{Lu5}, \cite{Lu08a} this is extended to the general case. 
This section is meant to give a first introduction into the formalism of 
Lusztig's book \cite{L1}. We try to present this here in a way which avoids 
the close discussion of the underlying technical apparatus, which is quite 
elaborate. 

Recall from \ref{dlgraph} the definition of the graph $\cG(G^F)$. We have seen
in Example~\ref{reguni} that every connected component of this graph contains 
at least one vertex which is labeled by a semisimple character of $G^F$. 
The first step is to clarify this situation. 

\begin{defn} \label{defsc} We set $\cS(G^F):=\{\rho_0\in\Irr(G^F)\mid 
\rho_0 \mbox{ semisimple}\}$. For $\rho_0\in \cS(G^F)$, we define 
\[\cE(\rho_0):=\{\rho\in\Irr(G^F)\mid \langle R_w^\theta,\rho\rangle\neq 0
\mbox{ and }\langle R_w^\theta, \rho_0\rangle\neq 0 \mbox{ for some 
$w,\theta$}\}.\]
The following result is already contained in the original article of 
Deligne--Lusztig \cite{DeLu} (see Proposition~\ref{ratser1} below for the 
case where $Z(G)$ is not connected).
\end{defn}

\begin{thm}[Deligne--Lusztig \protect{\cite[\S 10]{DeLu}}] \label{agree1}
Assume that $Z(G)$ is connected. 
\begin{itemize}
\item[(a)] Every connected component of $\cG(G^F)$ contains a unique
semisimple character. If $\rho_0\in\cS(G^F)$, then $\cE(\rho_0)$ is a
connected component of $\cG(G^F)$.
\item[(b)] We have a partition $\Irr(G^F)={\displaystyle \bigsqcup}_{\rho_0 
\in\cS(G^F)} \cE(\rho_0)$.
\item[(c)] The partition of $\Irr(G^F)$ in (b) corresponds precisely to the
partition into ``geometric conjugacy classes'' (as defined in 
\cite[\S 10]{DeLu}, \cite[\S 12.1]{C2}).
\end{itemize}
\end{thm}

\begin{proof} All that we need to know about ``geometric conjugacy classes''
(in addition to the implications in \ref{dlgraph}) is the following crucial
fact: if $Z(G)$ is connected, then every geometric conjugacy class of 
characters contains precisely one semisimple character; see 
\cite[10.7]{DeLu}, \cite[8.4.6]{C2}. Now we can argue as follows. We 
have already remarked in \ref{dlgraph} that every geometric conjugacy 
class of characters is a union of connected components of $\cG(G^F)$. We 
also know that every connected component of $\cG(G^F)$ has at least one 
vertex labeled by a semisimple character. So the above crucial fact shows 
that the partition of $\Irr(G^F)$ defined by the graph $\cG(G^F)$ 
corresponds precisely to the partition into geometric conjugacy classes.

In order to complete the proof, it now remains to show that $\cE(\rho_0)$,
for $\rho_0\in\cS(G^F)$, is a connected component of $\cG(G^F)$. This 
is seen as follows. Clearly, $\cE(\rho_0)$ is contained in a connected
component. Conversely, consider any $\rho\in\Irr(G^F)$ such that $\rho,
\rho_0$ belong to the same connected component. Let $w,\theta$ be such 
that $\langle R_w^\theta,\rho\rangle\neq 0$. By Example~\ref{reguni},
there exists some $\rho_0'\in\cS(G^F)$ such that $\langle R_w^\theta,
\rho_0'\rangle\neq 0$. Then $\rho_0,\rho_0'$ belong to the same
connected component, hence to the same geometric conjugacy class and, 
hence, $\rho_0=\rho_0'$, again by the crucial fact above.
\end{proof}

The next step consists of investigating a piece $\cE(\rho_0)$ in the
partition in Theorem~\ref{agree1}(b). The basic idea is as follows. We 
wish to associate with $\rho_0$ a subgroup $W'\subseteq W$ which should
be itself a Weyl group (i.e., generated by reflections of $W$); furthermore, 
there should be an induced automorphism $\gamma \colon W'\rightarrow W'$ 
such that we can apply the procedure in the previous section to form the 
set $\fX(W',\gamma)$. Then, finally, there should be a bijection $\cE
(\rho_0)\leftrightarrow \fX(W',\gamma)$ satisfying some further 
conditions. Now, in order to associate $W'$ with $\rho_0$, we have to
use in some way the underlying algebraic group~$G$.

\begin{rema} \label{abs50} For the discussion to follow, we need to fix
an embedding $\psi \colon k^\times \hookrightarrow \C^\times$.

This can be obtained as follows. Recall that $k=\overline{\F}_p$. Let $\A$ 
be the ring of algebraic integers in $\C$. Then one can find a surjective 
homomorphism of rings $\kappa \colon \A \rightarrow k$, which we fill fix 
from now on. (A definite choice of $\kappa$ could be made, for example,
using a construction of $k$ via Conway polynomials; see \cite[\S 4.2]{lupa}.
Another way to make this canonical is described by Lusztig 
\cite[\S 16]{twelve}.) Let $\mu_{p'}$ be the group of all roots of unity in 
$\C$ of order prime to~$p$. Then $\kappa$ restricts to an isomorphism 
$\mu_{p'}\stackrel{\sim}{\longrightarrow} k^\times$, and we let $\psi
\colon k^\times \stackrel{\sim}{\longrightarrow}\mu_{p'}\subseteq\C^\times$ 
be the inverse isomorphism. 
\end{rema}

%Furthermore, the exponential map induces a 
%group isomorphism
%\[ \exp\colon (\Q/\Z)_{p'}\stackrel{\sim}{\longrightarrow} \mu_{p'},
%\qquad n/m+\Z\mapsto \exp(2\pi i n/m),\]
%where $(\Q/\Z)_{p'}$ is the group of elements of $\Q/\Z$ of order prime to
%$p$. Hence, there is a unique group isomorphism $\iota\colon k^\times
%\rightarrow (\Q/\Z)_{p'}$ such that $\psi=\exp\circ \iota$.
%Thus, having fixed $\psi$, we can canonically
%identify $\Irr(\Gamma)=\Hom(\Gamma,k^\times)$, where $\Gamma$ is any 
%finite abelian group of order prime to~$p$. 

\begin{rema} \label{abs51} We can now relate $\Irr(T_0[w])$ (for $w\in W$) 
to elements in the character group $X=X(T_0)$ (see \ref{abs31}). The map 
$F'\colon T_0 \rightarrow T_0$, $t \mapsto \dot{w}F(t)\dot{w}^{-1}$, 
can also be regarded as a Frobenius map on $T_0$, and we have $T_0[w]=
T_0^{F'}$; see \cite[10.9]{St68} or \cite[1.4.13]{gema}. There is an 
induced map $X\rightarrow X$, $\lambda \mapsto \lambda \circ F'$, as in 
\ref{abs31}. Then, by \cite[(5.2.2)$^*$]{DeLu} (see also \cite[3.2.3]{C2}
or \cite[13.7]{DiMi2}), we have an exact sequence
\begin{equation*}
\{0\}\longrightarrow X\stackrel{F'{-}\id_X}{\xrightarrow{\hspace*{30pt}}} X 
\longrightarrow \Irr(T_0[w])\longrightarrow\{1\},\tag{a}
\end{equation*}
where the map $X\rightarrow \Irr(T_0[w])$ is given by restriction, followed 
by our embedding $\psi\colon k^\times \hookrightarrow \C^\times$. Now we 
proceed as follows, see Lusztig \cite[6.2]{L1}. Let $\theta\in
\Irr(T_0[w])$ and $n \geq 1$ be the smallest integer such that $\theta(t)^n
=1$ for all $t \in T_0[w]$. Then $p\nmid n$ and the values of $\theta$ lie 
in the image of our embedding $\psi\colon k^\times \hookrightarrow \C^\times$. 
So we can write $\theta=\psi\circ \bar{\theta}$ where $\bar{\theta}\colon 
T_0[w] \rightarrow k^\times$ is a group homomorphism. By (a), there 
exists some $\lambda_1\in X$ such that $\bar{\theta}$ is the restriction
of $\lambda_1$. Now $\theta^n=1$ and so $n\lambda_1$ is in the image of
the map $F'-\id_X\colon X\rightarrow X$. So there exists some $\lambda' 
\in X$ such that $\lambda'\circ F'-\lambda'=n\lambda_1$. Setting
$\lambda:=\underline{w}^{-1}(\lambda') \in X$, we obtain
\begin{equation*}
\lambda_1(t^n)=(n\lambda_1)(t)=\lambda'(F'(t)t^{-1})=\lambda(F(t)
\dot{w}^{-1}t^{-1}\dot{w})\qquad \mbox{for all $t \in T_0$}.\tag{b}
\end{equation*} 
Thus, we have associated with $(w,\theta)$ a pair $(\lambda,n)$ where 
$\lambda\in X$ and $n \geq 1$ is an integer such that $p \nmid n$ and (b) 
holds for some $\lambda_1\in X$ such that $\theta=\psi \circ 
\lambda_1|_{T_0[w]}$. (Note that Lusztig actually works with line
bundles $L$ over $G/B_0$ instead of characters $\lambda\in X$, but one can
pass from one to the other by \cite[1.3.2]{L1}.)
\end{rema}

\begin{rem} \label{xtens} In the setting of \ref{abs51}, the integer $n$
is uniquely determined by $(w,\theta)$, but $\lambda\in X$ depends on the
choice of $\lambda_1\in X$ such that $\bar{\theta}$ is the restriction 
of~$\lambda_1$. Note that, once $\lambda_1$ is chosen, then $\lambda$ is 
uniquely determined since the map $T_0\rightarrow T_0$, $t\mapsto F(t)
\dot{w}^{-1}t^{-1}\dot{w}$, is surjective (Lang's Theorem). Now, let $\mu_1
\in X$ also be such that $\bar{\theta}$ is the restriction of~$\mu_1$. Then
$\mu_1-\lambda_1$ is trivial on $T_0[w]$ and so, by \ref{abs51}(a), we 
have $\mu_1-\lambda_1=\nu\circ F'-\nu$ for some $\nu\in X$. But then 
$\mu:=\lambda+n\underline{w}^{-1}(\nu)$ is the unique element of $X$ such 
that \ref{abs51}(b) holds with $\lambda_1$ replaced by~$\mu_1$. Thus, we 
can associate with $(w,\theta)$ the well-defined element 
\begin{center}
$(\frac{1}{n}+\Z)\otimes \lambda\in (\Q/\Z)_{p'} \otimes X$,
\end{center}
where $(\Q/\Z)_{p'}$ is the group of elements of $\Q/\Z$ of order prime 
to~$p$ (see \cite[\S 4.1]{C2}, \cite[\S 5]{DeLu}, \cite[Chap.~13]{DiMi2},
\cite[8.4]{L1} for a further discussion of this correspondence). 
\end{rem}

%. Also note that the actions of $W$ and of $F$ on $X$ induce actions 
%on $(\Q/\Z)_{p'} \otimes X$. Then \ref{abs51}(b) implies that the $W$-orbit 
%of $\frac{1}{n}\otimes \lambda$ in $(\Q/\Z)_{p'} \otimes X$ is 
%$F$-stable 

\begin{rema} \label{abs52} Conversely, let us begin with a pair $(\lambda,n)$
where $\lambda \in X$ and $n \geq 1$ is an integer such that $p\nmid n$. 
Following Lusztig \cite[2.1]{L1}, we define $Z_{\lambda,n}$ to be the set
of all $w\in W$ for which there exists some $\lambda_w\in X(T_0)$ such 
that \ref{abs51}(b) holds, that is, 
\[ \lambda(F(t))=\lambda(\dot{w}^{-1}t\dot{w})
\lambda_w(t^n)\qquad \mbox{for all $t\in T_0$}.\]
Note that $\lambda_w$, if it exists, is uniquely determined by $w$ (since
$T_0=\{t^n\mid t\in T_0\}$). Assume now that $Z_{\lambda,n}\neq 
\varnothing$. Then, for any $w\in Z_{\lambda,n}$, the restriction of 
$\lambda_w$ to $T_0[w]$ is a group homomorphism
\[\bar{\lambda}_w\colon T_0[w]\rightarrow k^\times\qquad\mbox{such that}
\qquad \bar{\lambda}_w^n=1.\]
Using our embedding $\psi\colon k^\times \hookrightarrow \C^\times$, we 
obtain a character $\theta_w:=\psi\circ \bar{\lambda}_w\in \Irr(T_0[w])$, 
such that $\theta_w^n=1$. So each $w \in Z_{\lambda,n}$ gives rise to a 
virtual character $R_w^{\theta_w}$.  
\end{rema}

\begin{defn} \label{def52} Assume that $Z_{\lambda,n}\neq \varnothing$.
Following Lusztig \cite[2.19, 6.5]{L1}, we set 
\[ \cE_{\lambda,n}:=\{\rho \in \Irr(G^F)\mid \langle R_w^{\theta_w},\rho
\rangle \neq 0 \mbox{ for some $w \in Z_{\lambda,n}$}\}.\]
By \ref{abs51}, any $\rho \in \Irr(G^F)$ belongs to $\cE_{\lambda,n}$ for
some $(\lambda,n)$ as above. (Note that Lusztig assumes that $Z(G)$ is
connected, but the definition can be given in general.) 
\end{defn}

\begin{exmp} \label{expsl2a} (a) Let $n=1$ and $\lambda\colon T_0
\rightarrow k^\times$, $t\mapsto 1$, be the neutral element of~$X$.
Let $w\in W$ and set $\lambda_w:=\lambda$. Then the condition in
\ref{abs52} trivially holds. So, in this case, we have $Z_{\lambda,1}=W$
and $R_w^{\theta_w}=R_w^1$ for all $w\in W$. Consequently, $\cE_{\lambda,1}$
is precisely the set of unipotent characters of $G^F$.

(b) Let $G^F=\SL_2(\F_q)$ where $q$ is odd. As
in Example~\ref{expsl2}, we write $T_0=\{S(\xi)\mid \xi\in k^\times\}$ 
and $W=\{1,s\}$. Let $n=2$ and $\lambda\in X$ be defined by
$\lambda(S(\xi))=\xi$ for all $\xi\in k^\times$. Now note that, if $t \in 
T_0$, then $t^{q-1}=F(t)t^{-1}$ and $t^{q+1}=F(t)\dot{s}^{-1}t^{-1}
\dot{s}$. Thus, if we define $\lambda_1,\lambda_s\in X$ by 
\[ \lambda_1(S(\xi))=\xi^{(q-1)/2} \quad \mbox{and}\quad 
\lambda_s(S(\xi))=\xi^{(q+1)/2} \qquad\mbox{for all $\xi\in k^\times$},\]
then the condition in \ref{abs52} holds. So $Z_{\lambda,2}=W$. The 
corresponding characters of $T_0[1]$ and $T_0[s]$ are the unique
non-trivial characters of order~$2$. So we obtain 
\[ \cE_{\lambda,2}=\{\rho_0',\rho_0'',\pi_0',\pi_0''\}.\]
Thus, while $\{\rho_0',\rho_0''\}$ and $\{\pi_0',\pi_0''\}$ form 
different connected components of the graph $\cG(G^F)$, the above
constructions reveal a hidden relation among these four characters. (This 
hidden relation is precisely the relation of ``geometric conjugacy''; see
the references in Remark~\ref{xtens}.)
\end{exmp}

\begin{rema} \label{abs53} Now the situation simplifies when the centre
$Z(G)$ is connected. Assume that this is the case. Let $\lambda,n$ be as 
in \ref{abs52} such that $Z_{\lambda,n}\neq \varnothing$. Then, by 
\cite[1.8, 2.19]{L1}, there exists a unique element $w_1 \in W$ 
of minimal length in $Z_{\lambda,n}$ and we have 
\[ Z_{\lambda,n}=w_1W_{\lambda,n}\]
where $W_{\lambda,n}$ is a reflection subgroup with a 
canonically defined set of simple reflections $S_{\lambda,n}$.
Furthermore, $F$ induces an automorphism $\gamma \colon W_{\lambda,n}
\rightarrow W_{\lambda,n}$ such that $\gamma(S_{\lambda,n})=S_{\lambda,n}$
and $\gamma(y)=\sigma(w_1yw_1^{-1})$ for $y \in W_{\lambda,n}$; see 
\cite[2.15]{L1}. By \cite[3.4.1]{L1}, $\gamma$ induces an automorphism 
of the underlying root system. So we can apply the procedure in 
Section~\ref{secuni} to $(W_{\lambda,n},\gamma)$, and obtain
a corresponding set $\fX(W_{\lambda,n},\gamma)$. 
\end{rema}

\begin{thm}[Lusztig \protect{\cite[Main Theorem 4.23]{L1}}] \label{lu423} 
Assume that $Z(G)$ is connected. Let $\rho_0\in \cS(G^F)$ and $(\lambda,n)$ 
be a pair as above such that $\rho_0\in \cE_{\lambda,n}$ (cf.\ 
Definition~\ref{def52}). Then we have $\cE(\rho_0)=\cE_{\lambda,n}$ and 
there exists a bijection
\[ \cE(\rho_0)\stackrel{\sim}{\longrightarrow} \fX(W_{\lambda,n},
\gamma), \qquad \rho \mapsto \hat{x}_\rho,\]
such that $\langle R_{w_1y}^{\theta_{w_1y}},\rho\rangle=(-1)^{l(w_1)}
m_{\hat{x}_\rho,y}$ for $\rho\in\cE_{\lambda,n}$ and $y\in W_{\lambda,n}$,
where $Z_{\lambda,n}=w_1W_{\lambda,n}$ (see \ref{abs53}) and $m_{\hat{x},
y}$ are the numbers in \ref{abs45}(b) (with respect to 
$(W_{\lambda,n},\gamma)$).
\end{thm}

\begin{rem} \label{remdual} As explained in \cite[8.4]{L1}, a pair 
$(\lambda,n)$ as above can be interpreted as a semisimple element $s \in
G^*$ where $G^*$ is a group ``dual'' to $G$. There is a corresponding
Frobenius map $F^*\colon G^* \rightarrow G^*$ and then the conjugacy
class of~$s$ is $F^*$-stable. This actually gives rise to a bijection
(see also \cite[\S 10]{DeLu}):
\[\cS(G^F) \stackrel{\sim}{\longrightarrow} \{\mbox{$F^*$-stable 
semisimple conjugacy classes of $G^*$}\}.\]
Furthermore, the set $\fX(W_{\lambda,n},\gamma)$ parametrizes the unipotent 
characters of $C_{G^*}(s)^{F^*}$; note that $C_{G^*}(s)$ is connected 
since $Z(G)$ is connected (by a result of Steinberg; see \cite[4.5.9]{C2}).
For further details about ``dual'' groups, see \cite[Chap.~4]{C2}.
\end{rem}

\begin{rem} \label{sschar} Let $\cE(\rho_0)\stackrel{\sim}{\longrightarrow} 
\fX(W_{\lambda,n},\gamma)$ as in Theorem~\ref{lu423}. Then the semisimple
character $\rho_0 \in \cE(\rho_0)$ corresponds to the unique $\hat{x}_0\in
\fX(W_{\lambda,n},\gamma)$ such that $m_{\hat{x}_0,y}=1$ for all $y\in 
W_{\lambda,n}$ (see \ref{abs45}(c)). Indeed, by \cite[8.4.6]{C2}, 
\cite[10.7]{DeLu}, $\rho_0$ is uniform and it has an explicit expression 
as a linear combination of virtual characters $R_{T,\theta}$. As in 
\cite[14.47]{DiMi2}, this can be rewritten in the form 
\[\rho_0=(-1)^{l(w_1)}|W_{\lambda,n}|^{-1}\sum_{w \in Z_{\lambda,
n}} R_{w}^{\theta_{w}}.\]
Now let $\rho\in\cE(\rho_0)$ correspond to $\hat{x}_0$. By
Theorem~\ref{lu423}, we have $\langle R_{w}^{\theta_{w}},\rho
\rangle=(-1)^{l(w_1)}$ for all $w\in Z_{\lambda,n}$. As in \ref{abs45}(c), 
we find that $\langle \rho_0, \rho\rangle=1$ and, hence, $\rho=\rho_0$.
\end{rem}

\begin{rem} \label{allrt} For any $w'\in W$ and any $\theta'\in
\Irr(T_0[w'])$, the virtual character $R_{w'}^{\theta'}$ is equal to
one of the virtual characters in the set 
\[\{ R_w^{\theta_w}\mid w\in Z_{\lambda,n}
\mbox{ for some $(\lambda,n)$ as above}\}.\]
Indeed, let $\rho_0\in \cS(G^F)$ be such that $\langle R_{w'}^{\theta'},
\rho_0\rangle\neq 0$. Then choose $(\lambda,n)$ such that $Z_{\lambda,n}
\neq \varnothing$ and $\rho_0\in \cE_{\lambda,n}$. Using 
Remark~\ref{sschar}, we have $\langle R_{w'}^{\theta'}, R_w^{\theta_w}
\rangle\neq 0$ for some $w\in Z_{\lambda,n}$. So, finally, 
$R_{w'}^{\theta'}=R_w^{\theta_w}$ by Proposition~\ref{scform}.
\end{rem}

\begin{rema} \label{plan1} The above results lead to a general plan for 
classifying the irreducible characters of $G^F$, assuming that $Z(G)$ is 
connected. In a first step, one considers the dual group $G^*$ and 
determines the $F$-stable semisimple conjugacy classes of $G^*$. As in 
Remark~\ref{remdual}, this gives a parametrization of the set $\cS(G^F)$. 
Since each $\rho_0\in\cS(G^F)$ is uniform, one can even compute~---~at 
least in principle~---~the character values of $\rho_0$ by using the 
character formula for $R_w^\theta$ (mentioned in Remark~\ref{green}) and 
known results on Green functions \cite{S1}, \cite{S5a}. For a given 
$\rho_0 \in \cS(G^F)$, one then determines a corresponding pair 
$(W_{\lambda,n},\gamma)$ as above. The characters in $\cE(\rho_0)$ are 
parametrized by $\fX(W_{\lambda,n},\gamma)$ and we know the multiplicities 
of these characters in $R_w^\theta$ for all $w,\theta$; hence, we can 
also work out the corresponding degree polynomials. A large portion of 
this whole procedure (and also the one in \ref{plan2} below) can even be
put on a computer; see \cite{chev} and L\"ubeck \cite{lue} where explicit 
data are made available for various series of groups.
\end{rema}

%\begin{proof}  By \cite[14.47]{DiMi2}, the class function 
%$|W_{\lambda,n}|^{-1}\sum_{w \in W_{\lambda,n}} (-1)^{l(w)}
%R_w^{\hat{\lambda}_w}$ is $\pm $ an irreducible character of $G^F$. 
%Applying the duality operator $D_G$ in \cite[8.8]{DiMi2} and using 
%\cite[12.8]{DiMi2}, we also see that $\rho_{\lambda,n}$ is $\pm$ an 
%irreducible character. Now 
%\[ \rho_{\lambda,n}(1)=(-1)^{l(w_1)}|W_{\lambda,n}|^{-1}\sum_{w \in 
%W_{\lambda,n}} (-1)^{l(w)}q^{-N}[G^F:T_0[w]]\]
%\end{proof}

%%%%%%%%%%%%%%%%%%%%%%%%%%%%%%%%%%%%%%%%%%%%%%%%%%%%%%%%%%%%%%%%%%%%%%%
\section{Regular embeddings} \label{secj1}

Let us drop the assumption that $Z(G)$ is connected. Then we can find an 
embedding $G \subseteq \tG$, where $\tG$ is a connected reductive 
group over~$k$ such that 
\begin{itemize}
\item $G$ is a closed subgroup of $\tG$ and $G,\tG$ have the same derived 
subgroup; 
\item $Z(\tG)$ is connected;
\item there is a Frobenius map $\tF\colon \tG \rightarrow \tG$ whose 
restriction to $G$ equals~$F$.
\end{itemize}
Such an embedding is called a ``regular embedding''. (See 
\cite[5.18]{DeLu}, \cite{Lu5}.) This is the key tool to transfer results 
from the connected centre case to the general case. 

\begin{exmp} \label{regexp} (a) Assume that $Z(G)$ is connected and let 
$\Gder\subseteq G$ be the derived subgroup of $G$. Then, clearly,
$\Gder \subseteq G$ is a regular embedding. The standard example is given 
by $G=\GL_n(k)$ where $\Gder=\SL_n(k)$.

(b) Let $G=\SL_n(k)$ and suppose we did not know yet of the existence of 
$\GL_n(k)$. Then we can \textit{construct} a regular embedding $G\subseteq 
\tG$ as follows. Let
\[\tG:=\{(A,\xi)\in M_n(k) \times k^\times\mid \xi\det(A)=1\};\]
then $Z(\tG)=\{(\xi I_n,\xi^{-n})\mid \xi\in k^\times\}$ is connected and
$\dim Z(\tG)=1$. Identifying $A \in G$ with $(A,1)\in \tG$, we obtain
a regular embedding. A corresponding Frobenius map $\tF \colon \tG
\rightarrow\tG$ is defined by $\tF(A,\xi)=(F(A),\xi^q)$ if $G^F=\SL_n(q)$, 
and by $\tF(A,\xi)=(F(A),\xi^{-q})$ if $G^F=\mbox{SU}_n(q)$. This is just 
an example of the general construction mentioned above. For further
examples see \cite[\S 1.7]{gema}.
\end{exmp}

We now fix a regular embedding $G \subseteq \tG$. 

%By \cite[1.4]{Le78}, we have $\tG^{\tF}=G^F.\tT_0^{\tF}$ and 
%$T_0^F=G^F\cap \tT_0^{\tF}$. Furthermore, 
\begin{rema} \label{rege1} Here are some purely group-theoretical properties;
see Lehrer \cite[\S 1]{Le78}. First of all, one easily sees that $G^F$ is a 
normal subgroup of $\tG^{\tF}$ such that $\tG^{\tF}/G^F$ is abelian. Let 
$Z=Z(G)$ and $\tZ=Z(\tG)$. If $T\subseteq G$ is an $F$-stable maximal torus, 
then $\tT:=\tZ.T\subseteq \tG$ is an $\tF$-stable maximal torus; a similar 
statement holds for Borel subgroups. In particular, we obtain an $F$-stable 
split $BN$-pair $(\tT_0,\tB_0)$ in $\tG$ and the inclusion $N_G(T_0) 
\subseteq N_{\tG}(\tT_0)$ induces a canonical isomorphism between the Weyl 
group $W=N_G(T_0)/T_0$ of $G$ and the Weyl group $N_{\tG}(\tT_0)/\tT_0$ 
of~$\tG$. We have $\tG=G.\tZ$ but $G^F.\tZ^\tF \subsetneqq \tG^{\tF}$, in 
general. The gap is measured by the group $(Z/Z^\circ)_F$ (the largest
quotient of $Z/Z^\circ$ on which $F$ acts trivially). By \cite[1.2]{Le78},
\cite[1.7.6]{gema}, we have an exact sequence
\[ \{1\} \longrightarrow G^F.\tZ^{\tF}\longrightarrow \tG^{\tF}
\longrightarrow (Z/Z^\circ)_F.\]
Note that every irreducible character of $G^F.\tZ^F$ restricts irreducibly 
to $G^F$, so if some irreducible character of $\tG^{\tF}$ becomes
reducible upon restriction to $G^F$, then the splitting must happen
between $\tG^{\tF}$ and $G^F.\tZ^{\tF}$. 
%The crucial fact about regular embeddings is: 
\end{rema}

\begin{thm}[Lusztig] \label{multfree} Let $\trho\in \Irr(\tG^{\tF})$. 
Then the restriction of $\trho$ to $G^F$ is multiplicity-free, that is, we
have $\trho|_{G^F}=\rho_1+\ldots +\rho_r$ where $\rho_1,\ldots,\rho_r$ are
distinct irreducible characters of $G^F$.
\end{thm}

This statement appeared in \cite[Prop.~10]{Lu5} (see also \cite{LuW}), 
with an outline of the strategy of the proof; the details of this proof, 
which are surprisingly complicated, were provided much later in
\cite{Lu08a}. In the meantime, Cabanes and Enguehard also gave a proof 
in \cite[Chap.~16]{CE}.

Let us now consider the virtual characters $R_w^\theta$ of $G^F$. As 
discussed in \ref{rege1}, we can identify $W$ with the Weyl group 
$N_{\tG}(\tT_0)/\tT_0$ of $\tG$. So it makes sense to define 
\[ \tT_0[w]:=\{t \in \tT_0\mid \tF(t)=\dot{w}^{-1}t\dot{w}\}\subseteq \tT_0
\qquad \mbox{for $w \in W$},\]
where $\dot{w}\in N_G(T_0)$ also is a representative of $w$ in 
the Weyl group of $\tG$. We have $T_0[w]\subseteq \tT_0[w]$ and so we can 
restrict characters from $\tT_0[w]$ to $T_0[w]$.

\begin{lem} \label{rtres} Let $w\in W$ and $\ttheta\in \Irr(\tT_0[w])$.
Then $R_w^{\ttheta}|_{G^F}=R_w^\theta$, where $\theta\in \Irr(T_0[w])$ is 
the restriction of $\ttheta$ to $T_0[w]$.
\end{lem}

\begin{proof} As in Remark~\ref{remrt}, we write $R_w^\theta=R_{T_w,
\underline{\theta}}$ and, similarly, $R_w^{\ttheta}=R_{\tT_w,
\underline{\ttheta}}$. Under these identifications, $\underline{\theta}
\in \Irr(T_w^F)$ is the restriction of $\underline{\ttheta}\in
\Irr(\tT_w^\tF)$ to $T_w^F$. Then 
\[ R_w^{\ttheta}|_{G^F}=R_{\tT_w,\underline{\ttheta}}|_{G^F}=R_{T_w,
\underline{\theta}}=R_w^\theta\]
where the second equality holds by \cite[13.22]{DiMi2} (see also 
\cite[Lemma~1.4]{bs2}).
\end{proof}

Before we continue, we settle a point that was left open in \ref{abs32};
this will also be an illustration of how regular embeddings can be used to
transfer results on characters from the connected centre case to the
general case.
 
\begin{rem} \label{abs32a} Let $\trho\in \Irr(\tG^{\tF})$ and write
$\trho|_{G^F}=\rho_1+\ldots +\rho_r$ as in Theorem~\ref{multfree}.
\begin{itemize}
\item[(a)] The degree polynomials (see \ref{abs32}) are related by 
$\D_{\rho_i}=\frac{1}{r}\D_{\trho}$ (as it should~be).
\item[(b)] All $\rho_i$ have the same unipotent support, which is the
unipotent support of~$\trho$. In particular, $\trho\in\cS(\tG^\tF)$ if and 
only if $\rho_i\in\cS(G^F)$ for all~$i$.
\end{itemize}
\end{rem}

\begin{proof} (a) For $w\in W$ let $Q_w$ and $\tilde{Q}_w$ be the
corresponding Green functions (see Remark~\ref{green}) for $G^F$ and 
$\tG^{\tF}$, respectively. Then we can write $\D_{\rho_i}$ as
\[\D_{\rho_i}=\frac{1}{|W|}\sum_{w\in W} (-1)^{l(w)}
|T_0[w]| \langle Q_w,\rho_i\rangle \bq^{-N}\frac{|\G|}{|\T_w|}, \]
and we have a similar formula for $\D_{\trho}$. By Lemma~\ref{rtres}, $Q_w$ 
is the restriction of $\tilde{Q}_w$. Since the $\rho_i$ are conjugate
under $\tG^{\tF}$, this implies that $\langle Q_w,\rho_i \rangle$ does not
depend on~$i$ and so all $\D_{\rho_i}$ are equal. Thus, 
\[r\D_{\rho_i} =\D_{\rho_1}+\ldots +\D_{\rho_r}=\frac{1}{|W|}\sum_{w\in W}
(-1)^{l(w)} |T_0[w]| \langle Q_w,\trho|_{G^F}\rangle\,\bq^{-N}
\frac{|\G|}{|\T_w|},\]
for any fixed $i\in\{1,\ldots,r\}$. Next note that $G_\text{uni}=
\tG_{\text{uni}}$ and so 
\[ \Ind_{G^F}^{\tG^{\tF}}(\trho|_{G^F})(u)=|\tG^{\tF}:G^F|\trho(u)
\qquad \mbox{for all unipotent $u\in \tG^{\tF}$}.\]
Hence, using Frobenius reciprocity, we obtain 
$\langle Q_w,\trho|_{G^F}\rangle=|\tG^{\tF}:G^F|\langle \tilde{Q}_w,
\trho\rangle$. Finally, the formulae in Remark~\ref{abs31a} imply that
$|\tilde{\G}|/|\tilde{\G}_w|=|\G|/|\T_w|$ for any $w\in W$. This 
immediately yields that $r\D_{\rho_i}=\D_{\trho}$ for all~$i$.

(b) Let $\cO$ be an $F$-stable unipotent class of $G$; then $\cO$ 
also is an $\tF$-stable unipotent class in~$\tG$. Since the $\rho_i$ are 
$G^F$-conjugate, we conclude that $\AV(\rho_i,\cO)$ does not depend 
on~$i$. As in the proof of \cite[3.7]{GM1}, we then have 
\[r|A(u)|\AV(\rho_i,\cO)=|\tilde{A}(u)|\AV(\trho,\cO)\qquad\mbox{for 
any fixed~$i$},\]
where $A(u)=C_G(u)/C_G^\circ(u)$ and $\tilde{A}(u)=
C_{\tG}(u)/C_{\tG}^\circ(u)$ for $u\in \cO$. This clearly yields the 
statement about the unipotent supports of the~$\rho_i$. 
\end{proof}

\begin{rema} \label{multfree1} We now discuss some purely Clifford-theoretic 
aspects (cf.\ \cite[\S 2]{Le78}). Let $\Theta$ denote the group of all linear 
characters $\tta\colon \tG^{\tF}\rightarrow\C^\times$ with $G^F\subseteq 
\ker(\tta)$. Then $\Theta$ acts on $\Irr(\tG^{\tF})$ via $\trho\mapsto 
\tta\cdot \trho$ (usual pointwise multiplication of class functions).
If $\trho,\trho'\in\Irr(\tG^\tF)$, then either $\trho|_{G^F}$, 
$\trho'|_{G^F}$ do not have any irreducible constituent in common, or we 
have $\trho|_{G^F}= \trho'|_{G^F}$; the latter case happens precisely when 
$\trho'=\tta\cdot \trho$ for some $\tta\in\Theta$.
Given $\trho\in \Irr(\tG^{\tF})$, we denote the stabilizer of $\trho$ by
\[\Theta(\trho):=\{\tta\in \Theta\mid \tta\cdot \trho= \trho\}.\]
Now write $\trho|_{G^F}=\rho_1+\ldots +\rho_r$, as in 
Theorem~\ref{multfree}. Then 
\[ r=\langle \trho|_{G^F},\trho|_{G^F}\rangle=\big\langle \trho,
\mbox{Ind}_{G^F}^{\tG^{\tF}}(\trho|_{G^F})\bigr\rangle=\sum_{\tta\in
\Theta} \langle \trho,\tta\cdot \trho\rangle=|\Theta(\trho)|.\]
Furthermore, let again $Z=Z(G)$ and $\tZ=Z(\tG)$ as in \ref{rege1}. Then 
\[ \Theta(\trho)\subseteq \{\tta\in\Theta \mid \tZ^F \subseteq 
\ker(\tta)\},\]
since any element of $\tZ^F$ acts as a scalar in a representation 
affording $\trho$. Consequently, $r=|\Theta(\trho)|$ divides 
$|\tG^{\tF}:G^F.\tZ^F|$ and, hence, the order of $(Z/Z^\circ)_F$ (where
we use the exact sequence in \ref{rege1}). 
\end{rema}

\begin{rem} \label{csinv} The set $\cS(\tG^{\tF})$ is invariant under
the action of $\Theta$. This is clear by the definition in 
Example~\ref{reguni}; just note that $\tG_{\text{uni}}=G_{\text{uni}}$ and
$\tta(u)=1$ for all $\tta\in\Theta$ and all unipotent elements $u \in G^F$.
\end{rem}

\begin{lem} \label{disc0} Let $\trho_1,\trho_2\in\Irr(\tG^\tF)$ and $\tta
\in \Theta$. If $\trho_1$ and $\trho_2$ belong to the same connected 
component of $\cG(G^F)$, then so do $\tta\cdot\trho_1$ and $\tta\cdot 
\trho_2$. Thus, the action of $\Theta$ on $\Irr(\tG^\tF)$ induces a 
permutation of the connected components of $\cG(\tG^\tF)$ satisfying
\[\tta\cdot\cE(\trho_0)=\cE(\tta\cdot \trho_0)\qquad \mbox{for $\trho_0\in
\cS(\tG^\tF)$ (cf.\ Definition~\ref{defsc})}.\]
\end{lem}

\begin{proof} By Theorem~\ref{agree1}, we have $\trho_1\in\cE(\trho_0)$ 
for some $\trho_0\in\cS(\tG^\tF)$. By the definition of $\cE(\trho_0)$, we 
have $\langle R_w^\ttheta,\trho_1\rangle\neq 0$ and $\langle R_w^{\ttheta},
\trho_0\rangle\neq 0$ for some $w\in W$ and $\ttheta\in\Irr(\tT_0
[w])$. We write $R_w^\ttheta=R_{\tT_w,\underline{\ttheta}}$ as in 
Remark~\ref{remrt}. Now $\tta$ is a ``$p$-constant function'' on 
$\tG^{\tF}$ since $G^F\subseteq \ker(\tta)$; see \cite[7.2]{DiMi2}. So, 
by \cite[12.6]{DiMi2}, we have
\[\tta\cdot R_{\tT_w,\underline{\ttheta}}=R_{\tT_w,\tta\cdot 
\underline{\ttheta}}\]
where, on the right hand side, $\tta$ also denotes the restriction of 
$\tta$ to $\tT_w^F$. Also note that $R_{\tT_w,\tta\cdot 
\underline{\ttheta}}=R_w^{\ttheta'}$ for a unique $\ttheta'\in 
\Irr(\tT_0[w])$ (again, by Remark~\ref{remrt}). Hence, 
\[\langle R_w^{\ttheta'},\tta\cdot \trho_1\rangle=\langle \tta\cdot 
R_{\tT_w,\underline{\ttheta}}, \tta\cdot \trho_1\rangle=\langle R_{\tT_w,
\underline{\ttheta}}, \trho_1\rangle=\langle R_w^\ttheta,
\trho_1\rangle\neq 0\]
and, similarly, $\langle R_w^{\ttheta'},\tta\cdot \trho_0\rangle=\langle 
R_w^\ttheta, \trho_0\rangle\neq 0$. By Remark~\ref{csinv}, we have $\tta
\cdot \trho_0\in\cS(\tG^{\tF})$ and so $\tta\cdot \trho_1 \in\cE(\tta\cdot 
\trho_0)$, by the definition of $\cE(\tta\cdot \trho_0)$. We conclude that, 
if $\trho_1$ and $\trho_2$ belong to $\cE(\trho_0)$, then $\tta\cdot\trho_1$ 
and $\tta\cdot \trho_2$ belong to $\cE(\tta\cdot \trho_0)$. It remains 
to use the characterisation of connected components in Theorem~\ref{agree1}.
\end{proof}

Let $\trho_0\in \cS(\tG^{\tF})$. Then Lemma~\ref{disc0} shows that, in 
particular, the set $\cE(\trho_0)$ is preserved under multiplication 
with any $\tta\in \Theta(\trho_0)$. 

\begin{prop}[Lusztig \protect{\cite[14.1]{L1}, \cite[\S 11]{Lu5}}] 
\label{dics} Let $\trho_0\in\cS(\tG^{\tF})$ and $\trho\in\cE(\trho_0)$.
Let $O\subseteq \cE(\trho_0)$ be the orbit of $\trho$ under the action 
of $\Theta(\trho_0)$. Write 
\[ \trho|_{G^F}=\rho_1+\ldots +\rho_r\;\mbox{ where } \; \rho_1,\ldots,
\rho_r\in\Irr(G^F) \quad \mbox{(see Theorem~\ref{multfree})}.\]
Then $r=|\Theta(\trho)|$ and $\Theta(\trho)\subseteq \Theta(\trho_0)$.
Let $w\in W$ and $\theta\in\Irr(T_0[w])$. If $\theta$ is the restriction 
of some $\ttheta\in \Irr(\tT_0[w])$ such that $\langle R_w^{\ttheta},
\trho_0\rangle\neq 0$, then
\[ \langle R_w^\theta,\rho_i\rangle=\sum_{\trho'\in O} \langle 
R_w^{\ttheta},\trho'\rangle\qquad\mbox{for $1\leq i\leq r$};\] 
otherwise, we have $\langle R_w^\theta, \rho_i\rangle=0$ for $1\leq i\leq r$. 
\end{prop}

\begin{proof} By \ref{multfree1}, we have $r=|\Theta(\trho)|$. If $\tta\in 
\Theta(\trho)$, then $\trho=\tta\cdot \trho$ and $\tta\cdot \trho_0\in
\cS(\tG^{\tF})$ belong to the same connected component of $\cG(G^F)$ 
by Lemma~\ref{disc0}. Hence, we must have $\trho_0=\tta\cdot \trho_0$ 
by Theorem~\ref{agree1}. This shows that $\Theta(\trho)\subseteq 
\Theta(\trho_0)$. It remains to prove the multiplicity formula.
Let $w\in W$ and $\theta\in\Irr(T_0[w])$ be the restriction of some 
$\ttheta\in \Irr(\tT_0[w])$. By Lemma~\ref{rtres} and Frobenius 
reciprocity, we have 
\begin{equation*}
\langle R_w^\theta,\rho_i\rangle=\langle R_w^{\ttheta}|_{G^F},\rho_i\rangle=
\langle R_w^{\ttheta},\mbox{Ind}_{G^F}^{\tG^{\tF}}(\rho_i)\rangle=
\sum_{\trho'\in O'} \langle R_w^{\ttheta},\trho'\rangle\tag{$*$}
\end{equation*}
where $O'\subseteq \Irr(G^F)$ is the full orbit of $\trho$ under the
action of $\Theta$. Now assume that $\langle R_w^\ttheta,\trho_0\rangle
\neq 0$. Let $\trho'\in O'$ be such that $\langle R_w^{\ttheta},\trho'
\rangle\neq 0$; then $\trho'\in\cE(\trho_0)$. If we write $\trho'=\tta 
\cdot \trho$ where $\tta\in\Theta$, then Lemma~\ref{disc0} shows that 
$\trho'$ and $\tta \cdot \trho_0$ belong to the same connected component 
of $\cG(\tG^{\tF})$ and so $\tta\cdot\trho_0=\trho_0$, that is, $\trho'
\in O$. Thus, the desired formula holds in this case.

Finally, assume that $\langle R_w^\theta,\rho_i\rangle\neq 0$. Then we 
must show that $\ttheta$ can be chosen such that $\langle R_w^\ttheta,
\trho_0\rangle \neq 0$. Now, ($*$) shows that $\langle R_w^{\ttheta},
\trho' \rangle\neq 0$ for some $\trho'\in O'$. Let $\tta\in\Theta$ be such 
that $\trho'=\tta \cdot \trho$, and let $\tta^*$ be the complex conjugate 
of $\tta$.  Now write again $R_w^\theta=R_{T_w,\underline{\theta}}$ and 
$R_w^\ttheta=R_{\tT_w,\underline{\ttheta}}$, as in the proof of 
Lemma~\ref{disc0}. Then we have:
\[\langle R_{\tT_w,\tta^*\cdot \underline{\ttheta}},\trho \rangle=
\langle \tta^*\cdot R_{\tT_w,\underline{\ttheta}},\trho\rangle=
\langle R_{\tT_w,\underline{\ttheta}},\tta\cdot \trho\rangle=
\langle R_w^{\ttheta},\trho'\rangle \neq 0.\] 
Now define $\ttheta'\in\Irr(\tT_0[w])$ by the condition that 
$\underline{\ttheta'}=\tta^*\cdot \underline{\ttheta}$; then 
$\langle R_w^{\ttheta'},\trho \rangle\neq 0$. Since $\trho\in\cE(\trho_0)$,
we also have $\langle R_w^{\ttheta'},\trho_0 \rangle\neq 0$. Finally, since
$\underline{\theta}$ is the restriction of $\underline{\ttheta}$ to $T_w^F$,
we also have that $\underline{\theta}$ is the restriction of 
$\underline{\ttheta}'$ to $T_w^F$. So the restriction of $\ttheta'$ 
to $T_0[w]$ equals $\theta$, as required. 
\end{proof}

The crucial ingredient in the above result is the action of $\Theta$ on 
$\Irr(\tG^{\tF})$, and it would be very useful to describe this action in 
terms of the bijections in Theorem~\ref{lu423}. This is done in the further
parts of \cite{Lu5}, leading to the final statement of a ``Jordan 
decomposition'' in \cite[\S 12]{Lu5}. In this respect, the following 
example plays a special role; see part (a) of the proof of 
\cite[Prop.~8.1]{Lu5}. 

\begin{exmp}[Lusztig \protect{\cite[p.~353]{L1}}] \label{disexp} Assume 
that $p>2$ and $G$ is simple of simply-connected type $E_7$. Then $Z(G)$ 
has order~$2$ and so an irreducible character of $\tG^{\tF}$ either remains
irreducible upon restriction to $G^F$, or the restriction splits up as a
sum of two distinct irreducible characters. Now, there is a semisimple 
character $\trho_0\in \cS(\tG^{\tF})$ such that, in the setting of 
Theorem~\ref{lu423}, the group $W_{\lambda,n}$ is of type $E_6$ and 
$\gamma=\id$. By \ref{abs42}, the set $\fX^\circ (W_{\lambda,n},\gamma)
\hookrightarrow\fX(W_{\lambda,n},\gamma)$ contains two elements of the 
form $(\theta,3)$, $(\theta^2,3)$. Let $\trho_1,\trho_2$ be the 
corresponding irreducible characters in $\cE(\trho_0)$. The question is
whether these two characters are fixed by $\Theta(\trho_0)$ or not, and 
this turns out to be difficult to tell without knowing some extra
information, e.g., sufficiently many character values. In 
\cite[p.~353]{L1}, it is shown by a separate argument that the 
restrictions of $\trho_1,\trho_2$ to $G^F$ are reducible; so the conclusion 
is that these two characters must be fixed by~$\Theta(\trho_0)$.  
\end{exmp}

Here is an example where $O$ in Proposition~\ref{dics} 
has more than one element.

\begin{exmp} \label{disexp1} Let $G^F=\mbox{Sp}_4(\F_q)$, where $q$ is odd. 
Then we have a regular embedding $G\subseteq \tG$ where $\tG^{\tF}=
\mbox{CSp}_4(\F_q)$ is the conformal symplectic group. Here, $\tG^{\tF}/G^F
\cong \F_q^\times$ and $Z(G)$ has order~$2$. The character tables 
of $G^F$ and $\tG^{\tF}$ are known by Srinivasan \cite{Sr68} and Shinoda 
\cite{shin}, respectively. Let $\trho_0\in\cS(\tG^\tF)$ be one of the 
$\frac{1}{2}(q-1)$ semisimple characters of $\tG^\tF$ denoted by $\tau_1
(\lambda)$ in \cite[5.1]{shin}. We have 
\[\trho_0(1)=q^2+1\qquad\mbox{and}\qquad \Theta(\trho_0)=\langle \tta
\rangle\cong \Z/2\Z\]
where $\tta\in\Theta$ is the unique character of order~$2$. (In the setting
of Remark~\ref{remdual}, $\trho_0$ corresponds to a semisimple $s\in\tG^*$
such that the Weyl group of $C_{\tG^*}(s)$ is of type $A_1 \times A_1$.) 
By \cite[5.3]{shin}, the corresponding connected component of 
$\cG(\tG^\tF)$ contains exactly four irreducible characters. Using the 
notation in \cite[5.1]{shin}, we have 
\[ \cE(\trho_0)=\{ \tau_1(\lambda),\,\tau_2(\lambda),\,
\tau_2(\lambda^*),\, \tau_3(\lambda)\}\qquad (\trho_0=\tau_1(\lambda)),\]
where $\tau_2(\lambda)(1)=\tau_2(\lambda^*)(1)=q(q^2+1)$ and $\tau_3
(\lambda)(1)=q^2(q^2+1)$. Now we restrict these characters to $G^F$. 
Using the notation of \cite{Sr68}, we find that 
\[ \tau_1(\lambda)|_{G^F}=\theta_3+\theta_4,\quad
 \tau_3(\lambda)|_{G^F}=\theta_1+\theta_2,\quad 
\tau_2(\lambda)|_{G^F}=\tau_2(\lambda^*)|_{G^F}=\Phi_9,\]
where $\theta_1,\theta_2,\theta_3,\theta_4,\Phi_9\in\Irr(G^F)$. We conclude
that $\tau_2(\lambda^*)=\tta\cdot \tau_2(\lambda)$ and so 
\[O=\{\tau_2(\lambda),\tau_2(\lambda^*)\} \quad\mbox{is the 
$\Theta(\trho_0)$-orbit of $\tau_2(\lambda)$}.\]
Let us evaluate the formula in Proposition~\ref{dics} in some particular 
cases. First, by \cite[5.3]{shin}, there is some $\ttheta\in\Irr(\tT_0[1])$ 
such that 
\[ R_1^{\ttheta}=\tau_1(\lambda)+\tau_2(\lambda)+\tau_2(\lambda^*)+
\tau_3(\lambda).\]
Let $\theta\in \Irr(T_0[1])$ be the restriction of $\ttheta$. Then, by
Proposition~\ref{dics}, we have
\[ \langle R_1^\theta,\Phi_9\rangle=\langle R_1^\ttheta,\tau_2(\lambda)
\rangle+ \langle R_1^\ttheta,\tau_2(\lambda^*)\rangle=1+1=2,\]
(which is consistent with \cite[p.~191]{Sr94}). On the other hand, let 
$w_\alpha\in W$ be the simple reflection corresponding to the short simple 
root~$\alpha$. Then, again by \cite[5.3]{shin}, there is some $\ttheta'\in
\Irr(\tT_0[w_\alpha])$ such that 
\[ R_{w_\alpha}^{\ttheta'}=\tau_1(\lambda)-\tau_2(\lambda)+\tau_2(\lambda^*)-
\tau_3(\lambda).\]
Let $\theta'\in \Irr(T_0[w_\alpha])$ be the restriction of $\ttheta'$. 
Then, by Proposition~\ref{dics}, we have
\[ \langle R_{w_\alpha}^{\theta'},\Phi_9\rangle=\langle
R_{w_\alpha}^{\ttheta'}, \tau_2(\lambda) \rangle+ \langle 
R_{w_\alpha}^{\ttheta'},\tau_2(\lambda^*) \rangle=-1+1=0,\]
(which is consistent with \cite[p.~191]{Sr94}).  Thus, the terms in the
sum $\sum_{\trho'\in O} \langle R_w^\ttheta, \trho'\rangle$ in 
Proposition~\ref{dics} can actually take different values. (Similar 
examples also exist for the regular embedding $\SL_n(k) \subseteq 
\GL_n(k)$, as was pointed out by C.~Bonnaf\'e.)
\end{exmp}

\begin{rema} \label{plan2} The above results lead 
to a general plan for classifying the irreducible characters of $G^F$, 
assuming that the analogous problem for $\tG^{\tF}$ has been solved 
(see~\ref{plan1}). There is one further step to do: determine the action 
of $\Theta$ on $\Irr(\tG^{\tF})$. Recall from Remark~\ref{csinv} that 
$\cS(\tG^{\tF})$ is invariant under this action. Then we have a partition
\begin{center}
$\Irr(G^F)={\displaystyle\bigsqcup}_{\trho_0}\cE(G^F,\trho_0)$
\end{center} 
where $\trho_0$ runs over a set of representatives of the $\Theta$-orbits on
$\cS(\tG^\tF)$ and 
\[ \cE(G^F,\trho_0):=\{\rho\in\Irr(G^F)\mid \langle \trho|_{G^F},
\rho\rangle\neq 0\mbox{ for some $\trho\in\cE(\trho_0)$}\}\]
for $\trho_0\in \cS(\tG^\tF)$. Using Proposition~\ref{dics}, we can 
work out the number of irreducible constituents $\rho_i$ of any $\trho
\in\cE(\trho_0)$ and the multiplicities $\langle R_w^\theta,\rho_i\rangle$;
the degree polynomials of $\rho_i$ are determined by Remark~\ref{abs32a}. 
Note that the above union is indeed disjoint: if $\rho\in\Irr(G^F)$ and 
$\langle \trho|_{G^F},\rho \rangle\neq 0$, $\langle \trho'|_{G^F},\rho 
\rangle\neq 0$, where $\trho\in\cE(\trho_0)$ and $\trho'\in\cE(\trho_0')$,
then there exists some $\tta\in\Theta$ such that $\trho'=\tta\cdot\trho$ 
(see \ref{multfree1}) and so $\trho_0'=\tta\cdot \trho_0$ (see
Lemma~\ref{disc0} and Theorem~\ref{agree1}).
\end{rema}

In a somewhat different setting and formulation, part (b) of the following 
result appears in Bonnaf\'e \cite[Prop.~11.7]{Bo3}.

\begin{prop} \label{ratser1} Let $\trho_0\in\cS(\tG^\tF)$ and $\rho_0\in 
\Irr(G^F)$ be such that $\langle \trho_0|_{G^F},\rho_0\rangle\neq 0$.
\begin{itemize}
\item[(a)] We have $\rho_0\in\cS(G^F)$ (see Remark~\ref{abs32a}(b)).
\item[(b)] The set $\cE(G^F,\trho_0)$ (see \ref{plan2}) equals the set 
$\cE(\rho_0)$ (see Definition~\ref{defsc}).
\item[(c)] The partition of $\Irr(G^F)$ in \ref{plan2} corresponds 
precisely to the partition given by the connected components of $\cG(G^F)$. 
\end{itemize}
\end{prop}

%\begin{prop} \label{ratser1} Let $\trho_0\in\cS(\tG^\tF)$. Then the 
%characters in $\cE(G^F,\trho_0)$ (see \ref{plan2}) form a connected 
%component of $\cG(G^F)$. If $\rho_0\in \Irr(G^F)$ is such that $\langle 
%\trho_0|_{G^F},\rho_0\rangle \neq 0$, then $\rho_0\in\cS(G^F)$ is a 
%semisimple character and $\cE(G^F,\trho_0)=\cE(\rho_0)$. 
%\end{prop}

\begin{proof} First we show that the characters in a connected component of 
$\cG(G^F)$ are all contained in a set $\cE(G^F,\trho_0)$ as above. Indeed, 
let $\rho,\rho'\in\Irr(G^F)$ and assume that $\langle R_w^\theta,\rho\rangle
\neq 0$, $\langle R_w^\theta,\rho' \rangle\neq 0$ for some $w,\theta$. 
Choose any $\ttheta\in\Irr(\tT_0[w])$ such that $\theta$ is the 
restriction of~$\ttheta$. By Lemma~\ref{rtres} and Frobenius reciprocity, 
there exist $\trho,\trho'\in \Irr(\tG^\tF)$ such that the scalar products 
$\langle R_w^\ttheta,\trho\rangle$, $\langle R_w^\ttheta,\trho'\rangle$, 
$\langle \trho|_{G^F},\rho\rangle$, $\langle \trho'|_{G^F},\rho'\rangle$ 
are all non-zero. Then $\trho,\trho'$ belong to the same connected 
component of $\cG(\tG^\tF)$ and so $\trho, \trho'\in\cE(\trho_0)$ for 
some $\trho_0\in\cS(\tG^\tF)$ (see Theorem~\ref{agree1}). But then, by 
definition, $\rho, \rho'\in  \cE(G^F,\trho_0)$, as claimed. It now remains
to show that $\cE(G^F,\trho_0)\subseteq \cE(\rho_0)$. 
So let $\rho\in\cE(G^F,\trho_0)$. We must show that $\langle R_w^\theta,
\rho\rangle \neq 0$ and $\langle R_w^\theta,\rho_0 \rangle \neq 0$ for some 
pair $(w,\theta)$. For this purpose, as in the proof of 
\cite[Prop.~11.7]{Bo3}, we consider the linear combination 
$\sum_{\trho\in\cE(\trho_0)} \trho(1) \trho$. By \cite[7.7]{L0a}, that 
linear combination is a uniform function. So there exist $w_i\in W$ and 
$\ttheta_i\in \Irr(\tT_0[w_i])$ (where $1\leq i\leq n$ for some~$n$), such
that 
\[ \sum_{\trho\in\cE(\trho_0)} \trho(1)\trho=\sum_{1\leq i\leq n}
c_i R_{w_i}^{\ttheta_i} \qquad\mbox{where $c_i\in\C$, $c_i\neq 0$ 
for all $i$}.\]
Here, we can assume that $R_{w_i}^{\ttheta_i}$ and $R_{w_j}^{\ttheta_j}$
are orthogonal if $i\neq j$ (see Proposition~\ref{scform}). Let $\theta_i$ 
denote the restriction of $\ttheta_i$ to $T_0[w_i]$. Now $\rho\in
\cE(G^F,\trho_0)$ has a non-zero scalar product with the restriction of 
the left hand side of the above identity to~$G^F$. Hence, using 
Lemma~\ref{rtres}, there exists some $i$ such that 
$\langle R_{w_i}^{\theta_i},\rho\rangle \neq 0$. On the other hand, 
since $c_i\neq 0$, the above identity shows that $\langle 
R_{w_i}^{\ttheta_i},\trho\rangle \neq 0$ for some $\trho\in\cE(\trho_0)$. 
By Example~\ref{reguni}, there exists some $\trho_0'\in\cS(\tG^\tF)$ 
such that $\langle R_{w_i}^{\ttheta_i},\trho_0'\rangle\neq 0$. Then 
$\trho\in\cE(\trho_0')$ and so Theorem~\ref{agree1} shows that 
$\trho_0'=\trho_0$. Thus, we have $\langle R_{w_i}^{\ttheta_i},
\trho_0\rangle \neq 0$. Now Proposition~\ref{dics} yields that 
$\langle R_{w_i}^{\theta_i}, \rho_0\rangle=\langle R_{w_i}^{\ttheta_i},
\trho_0\rangle\neq 0$; note that the orbit of $\trho_0$ under the action 
of $\Theta(\trho_0)$ is just $\{\trho_0\}$. Thus, $\rho\in\cE(\rho_0)$
as claimed. In particular, $\rho,\rho_0$ belong to the same connected 
component of $\cG(G^F)$.
\end{proof}

\begin{rem} \label{ratser} As promised in \ref{dlgraph}, we can now finally 
clarify the relations between the various partitions of $\Irr(G^F)$ given 
in terms of 
\begin{itemize}
\item the connected components of $\cG(G^F)$, 
\item ``rational series'' and 
\item ``geometric conjugacy classes''. 
\end{itemize}
%(a) First of all, by Theorem~\ref{agree1} and Proposition~\ref{ratser1}, the 
%connected components of $\cG(G^F)$ are given by the sets $\cE(\rho_0)$ in 
%Definition~\ref{defsc}. (If $Z(G)$ is not connected, then it can happen
%that $\cE(\rho_0)=\cE(\rho_0')$ where $\rho_0\neq \rho_0'$ in $\cS(G^F)$;
%this happens precisely when both $\rho_0$ and $\rho_0'$ occur in the
%restriction of some $\trho_0\in \cS(\tG^\tF)$ to $G^F$.)
%
If $Z(G)$ is connected, then Theorem~\ref{agree1}(c) shows that the 
implications in \ref{dlgraph} are, in fact, equivalences; so the three 
partitions agree in this case.

Now assume that $Z(G)$ is not connected. Then Examples~\ref{expsl2} and
\ref{expsl2a} already show that ``geometric conjugacy classes'' may be 
strictly larger than the connected components of $\cG(G^F)$. By
Proposition~\ref{ratser1}, the latter are given by the sets $\cE(G^F,
\trho_0)$. In turn, the sets $\cE(G^F,\trho_0)$ are known to be equal 
to the ``rational series'' of characters of $G^F$. (This follows from 
Theorem~\ref{agree1}(c) and \cite[Prop.~11.7]{Bo3}, \cite[7.5]{L0a}.) 
Thus, in general, ``rational series'' are just given by the connected 
components of $\cG(G^F)$; in other words, the first implication in
\ref{dlgraph} is always an equivalence. (This also follows from 
Digne--Michel \cite[Theorem~14.51]{DiMi2}.) 
\end{rem}

%%%%%%%%%%%%%%%%%%%%%%%%%%%%%%%%%%%%%%%%%%%%%%%%%%%%%%%%%%%%%%%%%%%%%%%
\section{Character sheaves} \label{seccs}

Finally, we turn to the problem of computing the values of the
irreducible characters of $G^F$. As in \cite{L7}, it will be convenient 
to express this in terms of finding the base changes between various vector
space bases of $\mbox{CF}(G^F)$.

\begin{rema} \label{abs61}
The first basis to consider will be denoted by $\bB_0$; it consists of the 
characteristic functions $f_C\colon G^F\rightarrow \C$ of the various 
conjugacy classes $C$ of $G^F$. (Here, $f_C$ takes the value $1$ on $C$
and the value $0$ on the complement of $C$.) This basis is well-understood; 
see, e.g., the chapters on conjugacy classes in Carter's book \cite{C2} 
(and the references there). As a model example, see Mizuno's \cite{Miz}
computation of all the conjugacy classes of $G^F=E_6(\F_q)$. The second 
basis is, of course, $\bA_0=\Irr(G^F)$, the set of irreducible characters 
of $G^F$. Thus, the character table of $G^F$ is the matrix which expresses
the base change between $\bA_0$ and $\bB_0$. Note that the results discussed 
in the previous sections provide a parametrization of $\bA_0=\Irr(G^F)$ in
a way which is almost totally unrelated with the basis $\bB_0$. 
\end{rema}

\begin{rema} \label{abs62}
The third basis to consider will be denoted by $\bA_1$; it consists
of Lusztig's ``almost characters''. These are defined as certain explicit 
linear combinations of the irreducible characters of $G^F$. (The definition
first appeared in \cite[4.25]{L1}, assuming that $Z(G)$ is connected; see 
\cite{L12a} for the general case.) Almost characters are only well-defined 
up to multiplication by a root of unity, but we can choose a set $\bA_1$ 
of almost characters which form an orthonormal basis of $\mbox{CF}(G^F)$. 
The matrix which expresses the base change between $\bA_0=\Irr(G^F)$ 
and $\bA_1$ is explicitly known and given in terms of Lusztig's 
``non-abelian Fourier matrices''.
\end{rema}

\begin{rema} \label{abs63} Already in \cite[13.7]{L1}, Lusztig 
conjectured that there should be a fourth basis, providing a geometric 
interpretation of the almost characters. The theory of character sheaves 
\cite{L2} gives a positive answer to this conjecture. Character sheaves 
are certain simple perverse sheaves in the bounded derived category 
$\mathcal{D}G$ of constructible $\overline{\Q}_\ell$-sheaves
(in the sense of Beilinson, Bernstein, Deligne \cite{bbd}) on the algebraic
group $G$, which are equivariant for the action of $G$ on itself by 
conjugation. If $A$ is such a character sheaf, we consider its inverse
image $F^*A$ under the Frobenius map. If $F^*A \cong A$ in $\mathcal{D}G$, 
we choose an isomorphism $\phi \colon F^*A \stackrel{\sim}{\rightarrow} A$ 
and then define a class function $\chi_A\in \mbox{CF}(G^F)$, called 
``characteristic function'', by $\chi_A(g)=\sum_i (-1)^i \mbox{Trace}
(\phi,\cH_g^i(A))$ for $g \in G^F$, where $\cH_g^i(A)$ are the stalks 
at~$g$ of the cohomology sheaves of~$A$ (see \cite[8.4]{L2}). Such a 
function is only well-defined up to multiplication with a non-zero scalar. 
Let $\hat{G}^F$ be the set of character sheaves (up to isomorphism)
which are isomorphic to their inverse image under~$F$, and set
\[ \bB_1:=\{\chi_A \mid A \in \hat{G}^F\} \subseteq \mbox{CF}(G^F),\]
where, for each $A \in \hat{G}^F$, the characteristic function $\chi_A$
is defined with respect to a fixed choice of $\phi \colon F^*A 
\stackrel{\sim}{\rightarrow} A$. In \cite[\S 17.8]{L2}, Lusztig states
a number of properties of character sheaves. These include a ``multiplicity
formula'' \cite[17.8.3]{L2} (rather analogous to the Main Theorem~4.23 
of \cite{L1}), a ``cleanness condition'' and a ``parity condition'' 
\cite[17.8.4]{L2}, and a characterisation of arbitrary irreducible 
``cuspidal perverse sheaves'' on $G$ \cite[17.8.5]{L2}. In 
\cite[Theorem~23.1]{L2}, these properties were proved under a mild 
condition on~$p$; the conditions on~$p$ were later completely removed by
Lusztig \cite{L10}. (As a side remark we mention that a portion of the 
proof in \cite{L10} relies on computer calculations, as discussed in 
\cite[5.12]{some}.) As pointed out in \cite[3.10]{L10}, this allows us to 
state without any assumption on $G$, $p$ or $q$:
\end{rema}

\begin{thm}[Lusztig \protect{\cite{L2}, \cite{L10}}] \label{thm2} For
$A \in \hat{G}^F$, an isomorphism $\phi \colon F^*A 
\stackrel{\sim}{\rightarrow} A$ can be chosen such that the values of 
$\chi_A$ belong to a cyclotomic field and $\langle \chi_A,\chi_A\rangle
=1$. The corresponding set $\bB_1=\{\chi_A \mid A \in \hat{G}^F\}$ is an
orthonormal basis of $\operatorname{CF}(G^F)$.
\end{thm}

\begin{rema} \label{abs64} Consider the bases $\bA_1$ (almost characters)
and $\bB_1$ (character sheaves) of $\mbox{CF}(G^F)$. Lusztig 
\cite[13.7]{L1} (see also p.~226 in part II and p.~103 in part V of 
\cite{L2}) conjectured that the matrix which expresses the base change 
between $\bA_1$ and $\bB_1$ should be diagonal. For ``cuspidal'' objects in
$\bB_1$ this is proved in \cite[Theorem~0.8]{L7}, assuming that $p,q$ are 
sufficiently large. If $Z(G)$ is connected, then this conjecture was 
proved by Shoji \cite{S2}, \cite{S3}, under some mild conditions on~$p$. 
These conditions can now be removed since the properties in 
\cite[\S 17.8]{L2} mentioned in \ref{abs63} are known to hold in complete 
generality. Thus, we can state:
\end{rema}

\begin{thm}[Shoji \protect{\cite{S2}, \cite{S3}}] \label{thm3} Assume 
that $Z(G)$ is connected. Then the matrix which expresses the base change 
between the bases $\bA_1$ (see \ref{abs62}) and $\bB_1$ (see \ref{abs63}) 
is diagonal. Thus, there is a bijection $\bA_1\leftrightarrow \bB_1$ which 
is also explicitly determined in terms of the parametrizations of $\bA_1$ 
(see \cite[4.23]{L1}) and $\bB_1$ (see \cite[23.1]{L2}).
\end{thm}

\begin{rem} \label{cusp} There is a notion of ``cuspidal'' perverse
sheaves in $\mathcal{D}G$; see \cite[3.10, 7.1]{L2}. If $A \in \hat{G}^F$ is
cuspidal, then the ``cleanness condition'' \cite[17.8.4]{L2} implies that 
$\chi_A$ has the following property. The support $\{g \in G^F\mid \chi_A(g)
\neq 0\}$ is contained in $\Sigma^F$ where $\Sigma\subseteq G$ is an 
$F$-stable subset which is the inverse image of a single conjugacy class
in $G/Z(G)^\circ$ under the natural map $G \rightarrow G/Z(G)^\circ$. In 
particular, if $G$ is semisimple, then $\Sigma$ is just a conjugacy class 
in $G$. Explicit information about the sets $\Sigma$ can be extracted, for
example, from the proofs in \cite[\S 19--\S 21]{L2}.

Note that, in general, a class function on a finite group which has non-zero
values on very few conjugacy classes only, is typically a linear combination 
of many irreducible characters of the group. In contrast, it is actually 
quite impressive to see how the characteristic functions of $F$-stable 
cuspidal character sheaves are expressed as linear combinations of very
few irreducible characters of $G^F$. 
\end{rem}

\begin{exmp} \label{expcusp} (a) Let $G^F=\SL_2(\F_q)$ where $q$ is odd, 
as in Example~\ref{expsl2}. By \cite[\S 18]{L2}, there are two cuspidal 
character sheaves in $\hat{G}^F$. Their characteristic functions are equal 
(up to multiplication by a scalar of absolute value~$1$) to the two class 
functions in Remark~\ref{expsl2b}. The corresponding sets $\Sigma$ are 
the $G$-conjugacy classes of the elements $J$ and $-J$, respectively. 

(b) Let $G^F=\mbox{Sp}_4(\F_q)$ where $q$ is odd, as in 
Example~\ref{disexp1}. In Srinivasan's work \cite{Sr68}, apart from 
constructing characters by induction from various subgroups,
an extra function $\Gamma_1$ was constructed in \cite[(7.3)]{Sr68} by 
ad hoc methods in order to complete the character table. This function 
$\Gamma_1$ takes the values $q,-q,-q,q$ on the classes denoted $D_{31}$, 
$D_{32}$, $D_{33}$, $D_{34}$, respectively, and vanishes on all other 
classes of $G^F$. We have 
\begin{center}
$\Gamma_1=\frac{1}{2}(\theta_9+\theta_{10}-\theta_{11}-\theta_{12})\quad$
where $\quad\theta_9,\theta_{10},\theta_{11},\theta_{12}\in\fU(G^F)$.
\end{center}
On the other hand, by (a') in the proof of \cite[19.3]{L2}, there is a
unique cuspidal character sheaf $A_0\in\hat{G}^F$. In \cite[p.~192]{Sr94}, 
it is pointed out that $\Gamma_1$ is a characteristic function associated 
with~$A_0$.

(c) Let $G$ be simple of adjoint type $D_4$ and $F$ be such that 
$G^F={^3\!D}_4(\F_q)$, where $q$ is odd. By part (d) in the proof of 
\cite[19.3]{L2}, there are four cuspidal character sheaves, but only one 
of them is isomorphic to its inverse image under~$F$ (this is seen by an 
argument similar to that in the proof of \cite[20.4]{L2}); 
let us denote the latter one by $A_0$. The supporting set $\Sigma$ is the 
conjugacy class of an element $g=su=us\in G^F$ where $s$ has order $2$ and 
$u$ is regular unipotent in $C_G(s)^\circ$. The set $\Sigma^F$ splits into 
two classes in $G^F$, with representatives $su',su''$ as described in 
\cite[0.8]{Spa}. The characteristic function $\chi_{A_0}$ can be normalized
such that it takes values $q^2$ and $-q^2$ on $su'$ and $su''$, respectively
(and vanishes on all other classes of $G^F$). Using Spaltenstein's table 
\cite{Spa} of the unipotent characters of $G^F$, we see that (cf.\ 
\cite[p.~681]{Spa}):
\begin{center}
$\chi_{A_0}=\frac{1}{2}([\rho_1]-[\rho_2]+{^3\!D}_4[1]-{^3\!D}_4[-1])$.
\end{center}

(d) Further examples for $G$ of exceptional type are given by 
Kawanaka \cite[\S 4.2]{Kaw1}. These examples re-appear in the more general
framework of Lusztig \cite[\S 7]{Lu6}.
\end{exmp}

\begin{rema} \label{indcs} Finally, there is an inductive description of 
$\mbox{CF}(G^F)$ which highlights the relevance of the characteristic
functions of cuspidal characters sheaves. First, some definitions. As in 
\cite[7.2]{L0a}, a closed subgroup $L\subseteq G$ is called a 
\textit{regular subgroup} if $L$ is $F$-stable and there exists a 
parabolic subgroup $P\subseteq G$ (not necessarily $F$-stable) such that 
$L$ is a Levi subgroup of $P$, that is, $P=U_P\rtimes L$ where $U_P$ is 
the unipotent radical of $P$. By generalizing the construction of the
virtual characters $R_{T,\theta}$, Lusztig \cite{L0} (see also 
\cite[1.7]{L9}) defines a ``twisted'' induction 
\[ R_{L\subseteq P}^G\colon \mbox{CF}(L^F)\rightarrow \mbox{CF}(G^F),\]
which sends virtual characters of $L^F$ to virtual characters of $G^F$.
(A model of $R_{L\subseteq P}^G$ analogous to the model $R_w^\theta$ in 
Section~\ref{secdl} is described in \cite[6.21]{L2}.)

On the other hand, let $A_0\in \hat{L}^F$ be cuspidal and $\phi \colon 
F^*A_0 \stackrel{\sim}{\rightarrow} A_0$ be an isomorphism. To $(L,A_0,
\phi)$ one can associate a pair $(K,\tau)$ where $K$ is an object in 
$\mathcal{D}G$ (obtained by a geometric induction process from $A_0$) 
and $\tau\colon F^*K \stackrel{\sim}{\rightarrow} K$ is an isomorphism; 
see \cite[8.1]{L2}, \cite[1.8]{L9}. We have corresponding characteristic
functions $\chi_{A_0}\in\mbox{CF}(L^F)$ and $\chi_{K} \in \mbox{CF}(G^F)$. 
\end{rema}

\begin{thm}[Lusztig \protect{\cite[8.13, 9.2]{L9} $+$ Shoji \cite[\S 4]{S5}}] 
\label{thm4} Assume that $Z(G)$ is connected. Then, with the above notation,
the map $R_{L\subseteq P}^G\colon \operatorname{CF}(L^F)\rightarrow 
\operatorname{CF}(G^F)$ does not depend on $P$ and, hence, can be denoted
by $R_L^G$. If $(K,\tau)$ is associated with $(L,A_0,\phi)$ as above (where 
$A_0$ is cuspidal), then $\chi_K=(-1)^{\dim \Sigma} R_{L}^G(\chi_{A_0})$.
\end{thm}

In \cite[8.13, 9.2]{L9}, these statements are proved under a mild
condition on $p$ and for $q$ a sufficiently large power of~$p$. Again, 
the condition on $p$ can now be removed since the properties in 
\cite[\S 17.8]{L2} mentioned in \ref{abs63} (especially, the ``cleanness 
condition'') are known to hold in complete generality \cite{L10}; this even 
works when $Z(G)$ is not connected. Shoji \cite[\S 4]{S5} showed that one 
can also remove the assumption on $q$, when $Z(G)$ is connected. Then we
also have:

\begin{cor} \label{thm4a} Assume that $Z(G)$ is connected. Then
\[ \operatorname{CF}(G^F)=\big\langle R_L^G(\chi_{A_0})\mid L \subseteq G
\mbox{ regular and } A_0 \in \hat{L}^F \mbox{ cuspidal }\big\rangle_\C.\]
\end{cor}

\begin{proof} Since the properties in \cite[\S 17.8]{L2} hold in
complete generality, the set $\hat{G}$ coincides with the set 
of ``admissible complexes'' on $G$  in \cite[7.1.10]{L2}. Then the
above statement is contained in \cite[\S 10.4]{L2}; see also the
discussion in \cite[4.2]{S5a}. 
\end{proof}

\begin{rem} \label{fincs} In this picture, the general strategy for computing 
the character values of $\rho \in \Irr(G^F)$ is as follows
(see \cite[\S 4]{S5a} for further details).

If $A\in \hat{G}^F$, then $\chi_A$ can be written as an explicit
linear combination of induced class functions $R_L^G(\chi_{A_0})$ as in 
Corollary~\ref{thm4a} (see \cite[\S 10.4]{L2}). Now the values of 
$R_L^G(\chi_{A_0})$ can be determined by the character formula 
\cite[8.5]{L2}, which involves certain ``generalized Green functions''. 
The latter are computable by an algorithm which is described in 
\cite[\S 24]{L2}; see also Shoji \cite{S1}, \cite[\S 4.3]{S5a}. Thus, the 
values of the basis elements in $\bB_1$ can, at least in principle, be 
computed. Then one uses the transition from $\bB_1$ to $\bA_1$ in 
Theorem~\ref{thm3} and, finally, the transition from $\bA_1$ to $\bA_0=
\Irr(G^F)$ in \ref{abs62}. The remaining issue in this program is the
determination of the scalars in the diagonal base change in 
Theorem~\ref{thm3}. This problem appears to be very hard; it is not yet 
completely solved. But for many applications, one can already draw strong 
conclusions about character values without knowing these scalars 
precisely (an example will be given below).

The first successful realization of this whole program was carried out by
Lusztig \cite{L3}, where character values on unipotent elements are 
determined. See also Bonnaf\'e \cite{Bo3}, Shoji \cite{S5}, \cite{S5a}, 
\cite{S7}, Waldspurger \cite{wald} and the further references there.
\end{rem}

To close this section, we briefly mention an application of the above 
results to a concrete problem on the modular representation theory of $G^F$
in the non-defining characteristic case (see \cite{lupa} for general 
background on modular representations).

\begin{rema} \label{ibr} Let $\ell\neq p$ be a prime and $\mbox{CF}_{\ell'}
(G^F)$ be the space of all class functions $G^F_{\ell'} \rightarrow \C$,
where $G^F_{\ell'}$ denotes the set of elements $g \in G^F$ whose order 
is not divisible by~$\ell$. For any $f \in \mbox{CF}(G^F)$, we denote by 
$\breve{f}$ the restriction of $f$ to~$G^F_{\ell'}$. Let $\IBr_\ell(G^F) 
\subseteq \mbox{CF}_{\ell'}(G^F)$ be the set of irreducible Brauer
characters; these form a basis of $\mbox{CF}_{\ell'}(G^F)$. (The set
$\IBr_\ell(G^F)$ is in bijection with the isomorphism classes of 
irreducible representations of $G^F$ over an algebraically closed field
of characteristic~$\ell$.) For $\rho \in \Irr(G^F)$, we have
\[ \breve{\rho}=\sum_{\beta \in \IBr_{\ell}(G^F)} d_{\rho\beta} \, \beta
\qquad \mbox{where}\qquad  d_{\rho\beta}\in \Z_{\geq 0}.\]
The matrix $(d_{\rho\beta})_{\rho,\beta}$ is Brauer's ``$\ell$-modular 
decomposition matrix'' of $G^F$. We say that $\beta \in \IBr_\ell(G^F)$ is 
unipotent  if $d_{\rho\beta}\neq 0$ for some $\rho \in \fU(G^F)$. (See
also the general discussion in \cite[\S 3]{gehi2} for further comments 
and references.)
\end{rema}

\begin{table}[htbp] \caption{Number of unipotent irreducible $\ell$-modular 
Brauer characters} \label{tab1}
\begin{center}
$\begin{array}{|c|cccc|} \hline   
\;\;\mbox{Type}\;\;& \;\;\ell=2\;\; & \;\;\ell=3\;\; & \;\;\ell=5\;\;  
& \mbox{$\ell$ good} \\  \hline
\mbox{$G_2$}             &    9   &   8    &      & 10   \\ \hline 
%\mbox{$^3\!D_4$}         &    7   &        &      &  8   \\ \hline 
\mbox{$F_4$}             &   28   &  35    &      & 37   \\ \hline         
\mbox{$E_6$}, \mbox{$^2\!E_6$}&   27   &  28    &      & 30   \\ \hline
\mbox{$E_7$}             &   64   &  72    &      & 76   \\ \hline             
\mbox{$E_8$}             &   131  &  150   &  162 & 166  \\ \hline
\multicolumn{5}{c}{\text{(No entry means: same number as
for $\ell$ good)}}
\end{array}$
\end{center}
\end{table}

Assume now that $Z(G)$ is connected. In \cite[\S 6]{gehi2}, we found 
the number of unipotent $\beta\in \IBr_\ell(G^F)$ for $G$ of exceptional
type and ``bad'' primes $\ell$, under some mild restrictions on~$p$. 
(The table in \cite{gehi2} contained an error which was corrected in  
\cite[\S 4.1]{dghm}.) For ``good'' $\ell$ , those numbers were already 
known by \cite{bs1} to be equal to $|\fU(G^F)|$. For $G$ of classical 
type, see \cite[6.6]{gehi2} and the references there.

\begin{prop}[Cf. \protect{\cite[\S 6]{gehi2}}] \label{brauer} Assume that
$Z(G)$ is connected and $G/Z(G)$ is simple of type $G_2$, $F_4$, $E_6$, 
$E_7$ or $E_8$. If $\ell\neq p$, then the number of
unipotent $\beta\in \IBr_\ell(G^F)$ is given by Table~\ref{tab1}. 
Furthermore, we have the equality
\[ \langle \beta \mid  \beta \in \IBr_\ell(G^F) \mbox{ unipotent}\,
\rangle_\C=\langle \breve{\rho} \mid \rho \in \fU(G^F)\rangle_\C \quad
\subseteq \operatorname{CF}_{\ell'}(G^F).\]
\end{prop}

\begin{proof} If we knew the character tables of the groups in question,
then this would be a matter of a purely mechanical computation. Since those
tables are not known, the argument in \cite[\S 6]{gehi2} uses results on 
character sheaves but it requires three assumptions~A, B, C, as formulated
in \cite[5.2]{gehi2}. At the time of writing \cite{gehi2}, these assumptions 
were known to hold under some mild conditions on~$p$. These conditions can 
now be completely removed thanks to the facts that the ``cleanness 
condition'' holds unconditionally \cite{L10} and that Theorem~\ref{thm2}, 
\ref{thm3}, \ref{thm4} are valid as stated above. Otherwise, the argument 
remains the same as in \cite[\S 6]{gehi2}. 
\end{proof}

%%%%%%%%%%%%%%%%%%%%%%%%%%%%%%%%%%%%%%%%%%%%%%%%%%%%%%%%%%%%%%%%%%%%%%%
\section{Appendix: On uniform functions} \label{secapp}

The main purpose of this appendix is to provide a proof of 
Theorem~\ref{luconj} (Lusztig's conjecture on uniform functions). This 
may also serve as another illustration of the methods that are available 
in order to deal with a concrete problem. Recall from Section~\ref{secdl}
the definition of $R_w^\theta$. It will now be convenient to work with the 
model $R_{T, \theta}$ defined in \cite[\S 7.2]{C2}, \cite[2.2]{Lu2} (cf.\
Remark~\ref{remrt}). We will now further write $R_{T,\theta}$ as 
$R_T^G(\theta)$. Thus, for fixed $T\subseteq G$, we have a map $\theta
\mapsto R_T^G(\theta)$. Extending this linearly to all class functions, 
we obtain a linear map 
\[ R_T^G\colon \mbox{CF}(T^F)\rightarrow \mbox{CF}(G^F).\]
Let $G_{\text{uni}}$ be the set of unipotent elements of $G$. 
Then, as in Remark~\ref{green}, we obtain the Green function $Q_T^G\colon 
G_{\text{uni}}^F\rightarrow\Z$, $u\mapsto R_T^G(\theta)(u)$. 

By adjunction, there is a unique linear map $\Rst_T^G\colon
\mbox{CF}(G^F)\rightarrow \mbox{CF}(T^F)$ such that 
\[\langle R_T^G(f'),f\rangle=\langle f',\Rst_T^G(f)\rangle\qquad
\mbox{for all $f'\in \mbox{CF}(T^F)$ and $f\in \mbox{CF}(G^F)$}.\]
Then we have the following elegant characterisation of uniform functions.

%\begin{prop}[Lusztig \protect{\cite[2.15.2]{Lu2}}] \label{luuu}
%A  class function $f\in\operatorname{CF}(G^F)$ is uniform if and 
%only if 
%\[ \langle f,f\rangle=\sum_{(T,\theta)}\frac{\langle f,R_T^G(\theta)\rangle
%\langle R_T^G(\theta),f\rangle}{\langle R_T^G(\theta),R_T^G(\theta)\rangle}
%\]
%where the sum runs over all $G^F$-conjugacy classes of pairs
%$(T,\theta)$. 
%\end{prop}

\begin{prop}[Digne--Michel \protect{\cite[12.12]{DiMi2}}] \label{app1} 
A  class function $f\in\operatorname{CF}(G^F)$ is uniform if and 
only if 
\[ f=\frac{1}{|G^F|} \sum_{T\in \cT(G)}|T^F|(R_T^G\circ \Rst_T^G)(f),\]
where $\cT(G)$ denotes the set of all $F$-stable maximal tori $T\subseteq G$.
\end{prop}

\begin{rema} \label{step1} Let us fix a semisimple element $s_0\in G^F$ 
and assume that $H:=C_G(s_0)$ is connected. Then $H$ is $F$-stable,
closed, connected and reductive (see \cite[3.5.4]{C2}). Thus, $H$ itself 
is a connected reductive algebraic group and $F\colon H\rightarrow H$ is 
a Frobenius map. Let $T\subseteq G$ be an $F$-stable maximal torus such 
that $s_0\in T$. Then $T\subseteq H=C_G(s_0)$ and so we can form the 
Green function $Q_T^H\colon H_{\text{uni}}^F \rightarrow\Z$, $u\mapsto 
R_T^H(\theta)(u)$ (where $\theta$ is any irreducible character of $T^F$). 
\end{rema}

\begin{lem}[Cf.\ Lusztig \protect{\cite[25.5]{L2}}] \label{app3} 
Let $g\in G^F$ and write $g=su=us$ where $s\in G^F$ is semisimple, 
$u\in G^F$ is unipotent. Then, in the setting of \ref{step1}, we have
\[ \frac{1}{|T^F|}\sum_{\theta\in\Irr(T^F)} \theta(s_0)^{-1}R_T^G
(\theta)(g)=\left\{\begin{array}{cl} Q_T^H(u) &\mbox{ if $s=s_0$},
\\0&\mbox{ if $s$ is not $G^F$-conjugate to $s_0$}.\end{array}\right.\]
\end{lem}

\begin{proof} Consider the character formula for $R_T^G(\theta)$ in 
\cite[7.2.8]{C2}; we have
\[ R_T^G(\theta)(g)=\frac{1}{|H_s^F|} \sum_x \theta(x^{-1}sx)
Q_{xTx^{-1}}^{H_s}(u)\]
where $H_s:=C_G^\circ(s)$ is connected reductive (see again \cite[3.5.6]{C2})
and the sum runs over all $x\in G^F$ such that $x^{-1}sx\in T$ (and hence, 
$xTx^{-1}\subseteq H_s$). This yields that 
\[ \sum_{\theta\in\Irr(T^F)} \theta(s_0)^{-1}R_T^G(\theta)(g)=
\frac{1}{|H_s^F|} \sum_x \Bigl(\sum_{\theta\in\Irr(T^F)}\theta(s_0^{-1}x^{-1}
sx)\Bigr) Q_{xTx^{-1}}^{H_s}(u).\] 
Now the sum $\sum_{\theta \in\Irr(T^F)} \theta$ is the character of the
regular representation of $T^F$. Hence, for $x$ as above, we have
\[\sum_{\theta\in\Irr(T^F)}\theta(s_0^{-1}x^{-1}sx)=\left\{
\begin{array}{cl} |T^F| & \quad \mbox{if $s_0=x^{-1}sx$},\\ 0 & \quad
\mbox{otherwise}.\end{array}\right.\]
So, if $s$ is not $G^F$-conjugate to $s_0$, then 
\[ \sum_{\theta\in\Irr(T^F)} \theta(s_0)^{-1}R_T^G(\theta)(g)=0,\]
as desired. On the other hand, if $s=s_0$, then $H=H_s$ and we obtain 
\[ \sum_{\theta\in\Irr(T^F)} \theta(s_0)^{-1}R_T^G(\theta)(g)=
\frac{1}{|H^F|} \sum_{x\in H^F} |T^F|Q_{xTx^{-1}}^H(u).\] 
Since $Q_{xTx^{-1}}^H=Q_T^H$ for all $x\in H^F$, this yields the desired
formula. 
\end{proof}

\begin{rema} \label{step1a} Let $s_0\in G^F$ and $H=C_G(s_0)$ be as in
\ref{step1}. For any $f\in \mbox{CF}(H^F)$ such that $\{h\in H^F\mid f(h)
\neq 0\} \subseteq H_{\text{uni}}^F$, we can uniquely define a class
function $\hat{f} \in\mbox{CF}(G^F)$ by the requirement that 
\begin{equation*}
\hat{f}(g)=\left\{\begin{array}{cl} f(u) & \qquad \mbox{if $s=s_0$},\\
0 & \qquad \mbox{if $s$ is not $G^F$-conjugate to $s_0$},\end{array}\right.
\tag{a}
\end{equation*}
where $g\in G^F$ and $g=su=us$ with $s\in G^F$ semisimple, $u\in G^F$ 
unipotent (see Lusztig \cite[p.~151]{L2}). Thus, Lemma~\ref{app3} means
that 
\begin{equation*}
\hat{Q}_T^H=\frac{1}{|T^F|}\sum_{\theta\in\Irr(T^F)} 
\theta(s_0)^{-1}R_T^G(\theta),\tag{b}
\end{equation*}
for any $F$-stable maximal torus $T\subseteq H$. Hence, we deduce that
\begin{equation*}
f\mbox{ uniform} \qquad \Rightarrow\qquad \hat{f} \mbox{ uniform}.\tag{c}
\end{equation*}
Indeed, if $f$ is uniform, then $f$ can be written as a linear combination
of $R_T^H(\theta)$ for various $T,\theta$. Since the restriction of 
$R_T^H(\theta)$ to $H_{\text{uni}}^F$ is the Green function $Q_T^H$, we 
can write $f$ as a linear combination of Green functions $Q_T^H$ for 
various~$T$. Clearly, the map $f\mapsto\hat{f}$ is linear. Hence, (b) 
implies that $\hat{f}$ is uniform; so (c) holds.
\end{rema}

\begin{cor} \label{step1b} Assume that the derived subgroup $\Gder$ of 
$G$ is simply-connected. Then Theorem~\ref{luconj} holds.
\end{cor}

\begin{proof} The assumption implies that the centralizer of any
semisimple element is connected (see, e.g., \cite[3.5.6]{C2}). Hence,
we can apply the above discussion. Let $\cC$ be an $F$-stable conjugacy 
class of $G$ and $f_{\cC}^G \in \mbox{CF}(G^F)$ be the characteristic 
function of $\cC^F$. Let $g_0\in \cC^F$ and write $g_0=s_0u_0=u_0s_0$ 
where $s_0\in G^F$ is semisimple and $u_0\in G^F$ is unipotent. Let 
$H:=C_G(s_0)$ and set 
\[\cC':=\{u\in H_{\text{uni}} \mid s_0u\in \cC\}.\]
Then $\cC'$ is an $F$-stable unipotent class of $H$ and we denote by 
$f_{\cC'}^H\in\mbox{CF}(H^F)$ the characteristic function of $\cC'^F$. 
Let $f=f_{\cC'}^H$ and consider $\hat{f}$ as defined in \ref{step1a}(a).
One immediately checks that $\hat{f}=f_{\cC}^G$. Now, by 
\cite[Prop.~1.3]{aver}, $f=f_{\cC'}^H$ is uniform. Hence, \ref{step1a}(c) 
shows that $\hat{f}=f_{\cC}^G$ is uniform.
\end{proof}

The final step in the proof of Theorem~\ref{luconj} is to show that 
we can reduce it to the case where $\Gder$ is simply-connected. So let us 
now drop the assumption that $\Gder$ is simply-connected. Then we can find 
a surjective homomorphism of algebraic groups $\iota\colon G'\rightarrow G$
where
\begin{itemize}
\item $G'$ is connected reductive and $\Gder'$ is simply-connected,
\item the kernel $\ker(\iota)$ is contained in $Z(G')$ and is connected,
\item there is a Frobenius map $F'\colon G'\rightarrow G'$ such that $\iota
\circ F'=F\circ\iota$. 
\end{itemize}
(See \cite[p.~152]{L2} and \cite[1.7.13]{gema}.) Since
$\ker(\iota)$ is connected, a standard application of Lang's Theorem shows 
that $\iota$ restricts to a surjective map $G'^{F'}\rightarrow G^F$ which we 
denote again by~$\iota$. In this setting, we have:

\begin{prop}[Digne--Michel] \label{app2} Let $\iota\colon G'\rightarrow G$ 
be as above. Let $T\in\cT(G)$. Then $T'=\iota^{-1}(T)\in\cT(G')$ and we have:
\begin{itemize}
\item[(a)] $R_T^G(f')\circ \iota=R_{T'}^{G'}(f'\circ \iota)$ for any 
$f'\in \operatorname{CF}(T^F)$.
\item[(b)] $\Rst_T^G(f)\circ \iota=\Rst_{T'}^{G'}(f\circ \iota)$ for any 
$f\in \operatorname{CF}(G^F)$.
\end{itemize}
\end{prop}

\begin{proof} (a) This is a special case of a much more general result;
see \cite[13.22]{DiMi2} and also Bonnaf\'e \cite[\S 2]{Bo2}. (I thank
Jean Michel for pointing out the latter reference.)

(b) We have the following character formula (see \cite[12.2]{DiMi2}):
\[ \Rst_T^G(f)(t)=|T^F||H_t^F|^{-1}\sum_{u\in H_{t,\text{uni}}^F} 
Q_T^{H_t}(u)f(tu)\qquad\mbox{for any $f\in\mbox{CF}(G^F)$}\]
where $t\in T^F$ and $H_t:=C_G^\circ(t)$. Now, applying (a) with the 
trivial character of $T^F$, we obtain $Q_T^G\circ \iota=Q_{T'}^{G'}$.
Let $t'\in  T'^{F'}$ and $t=\iota(t')\in T^F$. Let $H_{t'}:=C_{G'}^\circ
(t')$. Then $\iota$ induces a bijection $H_{t',\text{uni}}^{F'}
\stackrel{\sim}{\longrightarrow} H_{t,\text{uni}}^F$. Since also 
$|T^F||H_t^F|^{-1}=|T'^F||H_{t'}^{F'}|^{-1}$, we obtain (b) using the 
character formulae for $\Rst_T^G$ and for $\Rst_{T'}^{G'}$. (See also 
\cite[\S 2]{Bo2}.)
\end{proof}

Now we can complete the proof of Theorem~\ref{luconj} as follows. Let $\cC$ 
be an $F$-stable conjugacy class of~$G$. Then $\Sigma:=\iota^{-1}(\cC)$ 
is a union of $F$-stable conjugacy classes of $G'$. Let $f_\Sigma^{G'}
\in\mbox{CF}(G'^{F'})$ be the characteristic function of $\Sigma^F$. 
Then $f_{\cC}^G\circ \iota=f_{\Sigma}^{G'}$ is uniform by 
Corollary~\ref{step1b}. So, using Propositions~\ref{app1} and~\ref{app2},
we obtain:
\begin{align*}
f_{\cC}^{G}\circ \iota&=\frac{1}{|G'^{F'}|} \sum_{T'\in\cT(G')}|T'^{F'}|
(R_{T'}^{G'}\circ \Rst_{T'}^{G'})(f_{\cC}^G\circ \iota)\\
&=\frac{1}{|G^{F}|} \sum_{T\in \cT(G)}|T^{F}| (R_{T}^{G}\circ \Rst_{T}^{G})
(f_{\cC}^G)\circ \iota,
\end{align*}
where we used that the map $\cT(G)\rightarrow \cT(G')$, $T\mapsto T':=
\iota^{-1}(T)$, is a bijection, such that $|G'^{F'}|/|T'^{F'}|=|G^F|/|T^F|$
for all $T\in\cT(G)$. Hence, we conclude that
\[f_{\cC}^{G}=\frac{1}{|G^{F}|} \sum_{T\in \cT(G)}|T^{F}| (R_{T}^{G}
\circ \Rst_{T}^{G}) (f_{\cC}^G)\]
and so $f_{\cC}^G$ is uniform, by Proposition~\ref{app1}. Thus,
Theorem~\ref{luconj} is proved.\qed

%Since $\ker(\iota)$ is connected, a standard application of Lang's Theorem
%shows that $f_\Sigma^{G'}=f_{\cC}^G\circ \iota$. Hence, we conclude 
%that $f_{\cC}^G\circ \iota$ is uniform.

\bigskip
\noindent
{\bf Acknowledgements.} I thank Pham Huu Tiep for asking me for some
precisions about the proof of Proposition~\ref{brauer} as stated above, 
where no conditions on $p$, $q$ are required. I also thank Toshiaki Shoji 
for clarifying comments about the contents of \cite{S2}, \cite{S3}, 
\cite{S5} and the status of Theorems~\ref{thm3} and \ref{thm4}. Finally,
I thank the referees for a careful reading of the manuscript and many
useful suggestions.

%%%%%%%%%%%%%%%%%%%%%%%%%%%%%%%%%%%%%%%%%%%%%%%%%%%%%%%%%%%%%%%%%%%%%%%

\end{document}